\documentclass[a4paper,12pt]{amsart}

\usepackage{amssymb, amsmath, amsthm, mathrsfs, braket, xspace}
\usepackage[margin=1.0in]{geometry}
\numberwithin{equation}{section}
\usepackage[colorlinks=true]{hyperref}

\usepackage{aliascnt}

\usepackage{mathtools}
\mathtoolsset{showonlyrefs=true} % For equation numbers

\usepackage[all,2cell]{xy}
\objectmargin+{1mm}
\labelmargin+{0.8mm}
\SelectTips{cm}{12}

\usepackage{enumitem}
\setlist[enumerate]{itemsep=0pt,label=$(\mathrm{\roman*})$, topsep=5pt}
\setlist[itemize]{itemsep=0pt, topsep=5pt, labelindent=\parindent, leftmargin=*}
\setlist[description]{itemsep=0pt, topsep=5pt, leftmargin=*}

\usepackage[OT2,T1]{fontenc}
\DeclareSymbolFont{cyrletters}{OT2}{wncyr}{m}{n}
\DeclareMathSymbol{\Sha}{\mathalpha}{cyrletters}{"58}

\newtheorem{thm}{Theorem}[section]

\newaliascnt{cor}{thm}
\newtheorem{cor}[cor]{Corollary}
\aliascntresetthe{cor}

\newaliascnt{lem}{thm}
\newtheorem{lem}[lem]{Lemma}
\aliascntresetthe{lem}

\newaliascnt{prop}{thm}
\newtheorem{prop}[prop]{Proposition}
\aliascntresetthe{prop}

\newaliascnt{conj}{thm}

\aliascntresetthe{conj}

\theoremstyle{definition}
\newaliascnt{dfn}{thm}

\aliascntresetthe{dfn}

\newaliascnt{rem}{thm}
\newtheorem{rem}[rem]{Remark}
\aliascntresetthe{rem}

\newaliascnt{prob}{thm}

\aliascntresetthe{prob}

\newaliascnt{ex}{thm}
\newtheorem{ex}[ex]{Example}
\aliascntresetthe{ex}

\newcommand{\ab}{\mathrm{ab}}

\DeclareMathOperator{\Aut}{Aut}

\newcommand{\bad}{\mathrm{bad}}

\newcommand{\C}{\mathbb{C}}

\newcommand{\Cor}{\operatorname{Cor}}
\newcommand{\Cf}{\textit{cf.}\;}
\DeclareMathOperator{\Coker}{Coker}

\renewcommand{\div}{\mathrm{div}}

\newcommand{\dimFp}{\dim_{\Fp}}
\renewcommand{\d}{\partial}
\newcommand{\dbar}{\ol{\d}}
\newcommand{\dt}{\d^{\,t}}

\newcommand{\dX}{\d_X}

\newcommand{\ds}{\displaystyle}

\newcommand{\Ehat}{\widehat{E}}

\newcommand{\Ebar}{\ol{E}}
\newcommand{\Evbar}{\ol{E}_v}

\newcommand{\F}{\mathbb{F}}
\newcommand{\Fbar}{\ol{F}}
\newcommand{\Fp}{\F_p}
\newcommand{\Fl}{\F_l}
\newcommand{\Fv}{\F_v}
\newcommand{\Frob}{\operatorname{Frob}}

\newcommand{\FK}{\F_K}
\newcommand{\fin}{\mathrm{fin}}

\newcommand{\geo}{\mathrm{geo}}

\DeclareMathOperator{\Gal}{Gal}
\newcommand{\Gm}{\mathbb{G}_{m}}

\newcommand{\good}{\mathrm{good}}

\newcommand{\Hom}{\operatorname{Hom}}

\newcommand{\isomto}{\xrightarrow{\simeq}}
\renewcommand{\Im}{\operatorname{Im}}

\DeclareMathOperator{\Jac}{Jac}
\newcommand{\Jbar}{\ol{J}}

\newcommand{\Jvbar}{\ol{J}_v}
\newcommand{\JKbar}{\Jbar_K}

\DeclareMathOperator{\Ker}{Ker}

\newcommand{\loc}{\operatorname{loc}}
\newcommand{\locbar}{\ol{\loc}}
\newcommand{\locpbar}{\locbar_p}

\newcommand{\m}{\mathfrak{m}}

\newcommand{\ol}[1]{\overline{#1}}
\renewcommand{\O}{\mathcal{O}}

\newcommand{\Pbar}{\overline{P}}

\newcommand{\piab}{\pi_1^{\ab}}
\newcommand{\piabX}{\piab(X)}

\newcommand{\piabXgeo}{\piabX^{\geo}}

\newcommand{\plim}{\varprojlim}

\newcommand{\pmat}[1]{\begin{pmatrix}
	#1
\end{pmatrix}}

\newcommand{\Q}{\mathbb{Q}}
\newcommand{\Qp}{\Q_p}
\newcommand{\Ql}{\Q_l}

\newcommand{\Qbar}{\ol{Q}}

\newcommand{\R}{\mathbb{R}}

\newcommand{\res}{\operatorname{res}}
\newcommand{\ram}{\mathrm{ram}}
\newcommand{\red}{\mathrm{red}}

\newcommand{\sn}{\smallskip\noindent}
\renewcommand{\sp}{\operatorname{sp}}
\newcommand{\sep}{\mathrm{sep}} 
\newcommand{\surj}{\twoheadrightarrow}
\newcommand{\ssm}{\smallsetminus}
\newcommand{\Spec}{\operatorname{Spec}}

\newcommand{\tor}{\mathrm{tor}}

\newcommand{\VX}{V(X)}

\newcommand{\wh}[1]{\widehat{#1}}

\newcommand{\X}{\mathscr{X}}

\newcommand{\XKbar}{\ol{X}\!_K}

\newcommand{\Z}{\mathbb{Z}}
\newcommand{\Zhat}{\wh{\Z}}
\newcommand{\Zp}{\Z_p}

\title[a Hasse principle for $GL_2(\Fp)$ and Bloch's exact sequence]{a Hasse principle for $GL_2(\Fp)$ and Bloch's exact sequence for elliptic curves over number fields}

\author[T. Hiranouchi]{Toshiro Hiranouchi}\address[T. Hiranouchi]{
Department of Basic Sciences, Graduate School of Engineering, 
Kyushu Institute of Technology, 
1-1 Sensui-cho, Tobata-ku, Kitakyushu-shi, 
Fukuoka 804-8550 JAPAN}
\email{hira@mns.kyutech.ac.jp}


\begin{document}
\pagenumbering{arabic}
\date{\today}

\begin{abstract}
	We investigate the higher Chow groups, specifically 
$SK_1(E)$ for elliptic curves $E$
over number fields $F$. 
Focusing on the kernel $V(E)$ of the norm map 
$SK_1(E)\to F^{\times}$, we analyze its mod $p$ structure. 
We provide conditions, based on the mod $p$ Galois representations associated to $E$, 
under which the torsion subgroup of $V(E)$ is infinite. 
%Furthermore, for an elliptic curves $E$ over $\Q$, 
%we examine the local components of 
%$V(E)$ at both finite and infinite places . 
%Applying these results to semi-stable elliptic curves $E$ over $\Q$, we offer insights into the structure of 
%$V(E)$ more precisely. 
\end{abstract}

\maketitle

\section{Introduction}
Let $X$ be a smooth projective curve defined over a number field $F$.
The higher Chow group $CH^{2}(X,1)$ of $X$ can be written as the cokernel of the tame symbol map $\dt_{F(X)}$: 
\[
	CH^{2}(X,1) \simeq \Coker\left(\dt_{F(X)} \colon K_2^M(F(X)) \to \bigoplus_{x\in X_0} F(x)^{\times}\right),
\]
where $K_2^M(F(X))$ denotes the Milnor $K$-group of the function field $F(X)$ of $X$ 
and $F(x)$ is the residue field at a closed point $x\in X_0$ 
(\Cf\cite[Thm.~3]{Kat86b}). 
Following \cite{Blo81}, 
we write $SK_1(X)$ for $CH^{2}(X,1)$. 
This abelian group plays a significant role in the higher dimensional class field theory of S.~Bloch \cite{Blo81}. K.~Kato and S.~Saito \cite{KS83b}. 
To investigate the structure of $SK_1(X)$, 
we consider the kernel   
\[
V(X)=\Ker\left(f_{\ast}\colon SK_1(X)\to F^\times\right), 
\]
where 
$f\colon X\to \Spec(F)$ is the 
structure morphism.
Since the structure of $F^{\times}$ is well understood due to the classical finiteness theorem for the class group $\operatorname{Cl}(\O_F)$ of 
the ring of integers $\O_F$ of $F$ and the structure theorem for the unit group $\O_F^{\times}$, 
we focus on $V(X)$. 
Bloch conjectured that $V(X)$ is a torsion group (\Cf\cite[Remark 1.24]{Blo81}). 
%To our knowledge the main evidence for this conjecture in the literature is due to Raskind who proved that $V(X)\otimes_{\Z} \Q/\Z = 0$ (\cite[Cor.~3.8]{Ras90}). 
%for a base change $X_{\Fbar} = X\otimes_F \Fbar$, 
%there is a short exact sequence 
%\[
%0 \to V(X)_{\tor} \to V(X) \to V(X_{\ol{F}})^{G_F}\to 0
%\]
%where the $G_F = \Gal(\ol{F}/F)$ fixed subgroup $V(X_{\Fbar})^{G_F}$ is uniquely divisible, 
%and $V(X)_{\tor}$ is the torsion subgroup of $V(X)$. 
%the group the quotient $V(X)/V(X)_{\tor}$ is uniquely divisible (that is, divisible and torsion free), where $V(X)_{\tor}$ is the torsion subgroup of $V(X)$. 
The aim of this note is to investigate the structure of the torsion subgroup $V(E)_{\tor}$ of $V(E)$ 
for an elliptic curve $E$ over $F$.  

It is known that $V(E)$ is isomorphic to the Somekawa $K$-group $K(F;E,\Gm)$ associated to $E$ and the multiplicative group $\Gm$ (\cite{Som90}). 
By replacing $E$ with $\Gm$, the Somekawa $K$-group  
$K(F;\Gm,\Gm)$ is isomorphic to the Milnor $K$-group $K_2^M(F)$ of the field $F$. 
The tame symbol map 
\[
	\dt_F\colon K_2^M(F) \to \bigoplus_{v\colon\text{finite place of $F$}} \F_v^{\times} 
\]
is surjective. 
Here, $\F_v$ is the residue field of $F$ at a finite place $v$ of $F$. 
The kernel $\Ker(\dt_F)$ %which is often called the \textbf{tame kernel} 
coincides with the algebraic $K$-group $K_2(\O_F)$ of the ring of integers $\O_F$ of $F$ 
and is investigated by many authors 
(see, e.g.\ \cite[Sect.~5.2]{Wei05}). 
In particular, the kernel $\Ker(\dt_F) = K_2(\O_F)$ 
is finite and 
%coincides with the algebraic $K$-group $K_2(\O_F)$ of the ring of integers $\O_F$ and 
is related to the order of the ideal class group of some number field. 
%(for the structure of $\Ker(\dt_F)/p\Ker(\dt_F)$, see \cite{Keu}). % Geijsberts Chap.~II, Sect.~3 
In the case $F=\Q$, more precisely, 
%$K_2(\Z) \simeq \Z/2\Z$ and hence 
the following split exact sequence exists: 
\[
0 \to \Z/2\Z \to K_2^M(\Q) \xrightarrow{\dt_\Q}  \bigoplus_{l\colon\text{prime}}\F_l^{\times}\to 0,
\]
where $l$ runs through the set of all prime numbers
 (\cite[Chap.~IX, Sect.~2]{FV02}).%, Geijsberts, Thm.~4.2, \cite[6.5.1]{Wei05} ****?).

To study the structure of the group $V(E) \simeq K(F;E,\Gm)$  
for an elliptic curve $E$ over $F$, 
we consider a map   
\[
\partial_E\colon V(E) \to \bigoplus_{v\colon\text{finite, good}}\Evbar(\F_v) 
\]
induced from the boundary map 
\[
\d_E\colon SK_1(E)\simeq CH^{2}(E,1)\to \bigoplus_{v\colon \text{finite, good}}CH^{1}(\Evbar,0) = \bigoplus_v CH_0(\Evbar)
\]
of the higher Chow group of $E$ (see \autoref{sec:cft} for the definition). 
Here, $v$ runs through the set of finite places of $F$ at which $E$ has good reduction 
and $\Evbar$ is the reduction of $E$ at $v$. 
Let $G_F = \Gal(\ol{F}/F)$ be the absolute Galois group of $F$ 
and $E[p]_{G_F}$ the maximal $G_F$-coinvariant quotient of the $p$-torsion points $E[p]$. 
It can be seen $E[p]_{G_F}$ is involved in the mod $p$ structure of $V(E)$ by combining 
the two local-global principles below: 
\begin{itemize}
	\item A Hasse principle for the cohomology group $H^1(G,M)$ of a subgroup $G$ of $GL_2(\Fp)$ 
due to Ramakrishnan (see \autoref{prop:Ram}), and 
	\item The exact sequence of Bloch for $V(E)$ (see \autoref{prop:KS} (ii)). 
\end{itemize} 

Our first main result is the following:

\begin{thm}[\autoref{thm:pdiv}]
\label{thm:intro}
Let $E$ be an elliptic curve over a number field $F$ 
and $p$ a rational prime. 
If $E[p]_{G_F} \neq 0$, then  
	the kernel and the cokernel of the 
	map 
	\[
	\dbar_{E,p}\colon V(E)/pV(E) \to  \bigoplus_{v\colon\mathrm{finite,\, good}}\Evbar(\Fv)/p\Evbar(\F_v)
	\]
	induced from $\d_E$ 
	are finite. 
\end{thm}

Using Raskind's theorem on $V(E)$, 
we have $V(E)_{\tor}/pV(E)_{\tor}\simeq V(E)/pV(E)$ (\autoref{lem:VXtor}). 
The theorem above implies 
$\dimFp(V(E)_{\tor}/pV(E)_{\tor}) = \infty$ for some $p$ if $E[p]_{G_F}\neq 0$ 
(\Cf \autoref{rem:infty}).
 
A prime $p$ satisfies $E[p]_{G_F}\neq 0$ if and only if 
$p$ is a prime divisor of the abelian geometric fundamental group 
$\piab(E)^{\geo} := \Ker(\piab(E)\to G_F^{\ab})$ which is known to be finite 
by Katz-Lang \cite{KL81} (see also \autoref{prop:KS}). 
For example, 
if the mod $p$ Galois representation 
$\rho_{E,p}\colon G_F\to \Aut(E[p])$ associated to $E[p]$ 
is surjective, 
then $E[p]_{G_F}=0$ (\autoref{lem:quot}). 
%By \cite[Sect.~4.2, Thm.~2]{Ser72}, 
%if $E$ does not have potentially complex multiplication, 
%then the representation $\rho_{E,p}$ is surjective for all but finitely many prime numbers $p$. 
%Consequently, our theorem can be applied to the remaining finitely many primes $p$. 
%cf. \cite{RV01} Reverter-Vila Images of mod p Galois, Intro.

In the case where $F=\Q$, 
the kernel and the cokernel of the map $\dbar_{E,p}$ 
is described by the local terms $V(E_l)/pV(E_l)$ for the \emph{bad primes} $l$, 
where $E_l := E\otimes_\Q\Ql$. 

\begin{thm}[\autoref{thm:EQ}]
\label{thm:introEQ}
	Let $E$ be an elliptic curve over $\Q$. 
	If $E[p]_{G_{\Q}} \neq 0$ for some odd prime $p$, 
	then 
	there is an exact sequence 
	\[
	0 \to \Ker(\dbar_{E,p})\to  \bigoplus_{l\colon \mathrm{bad}}V(E_{l})/pV(E_l)  \\
		\to \Z/p\Z \to \Coker(\dbar_{E,p})\to 0
	\]
	of finite dimensional $\Fp$-vector spaces, where $l$ runs through the set of primes $l$ at which 
	$E$ has bad reduction. 
\end{thm}

%It is known that if an elliptic curve $E$ over $\Q$ is semi-stable, 
%then $\rho_{E,p}$ is surjective for $p\ge 11$ (\cite[Thm.~4]{Maz78}). 
% potentially CM => potentially good reduction 
The local term $V(E_l)/pV(E_l)$ can be computed by 
using the Hilbert symbol when $E$ has multiplicative reduction at $l$ (\Cf \autoref{lem:mult}).  
For this reason, we study $\Ker(\dbar_{E,p})$ and $\Coker(\dbar_{E,p})$ more precisely 
for a semi-stable elliptic curve $E$ over $\Q$.  
For an odd prime $p$, 
if the mod $p$ Galois representation $\rho_{E,p}\colon G_\Q \to \Aut(E[p])$ is not surjective, then 
one of the three conditions below hold (\cite[Prop.~1]{Ser96}):
%We deal with the cases where non-trivial rational $p$-torsion point $E(\Q)[p]$ is an obstacle to $\rho_{E,p}$ being surjective.
%For an elliptic curve $E$ over $\Q$ and an odd prime $p$, 
%$E[p]_{G_\Q} \neq 0$ holds if one of the following conditions 
\begin{enumerate}
	\item[$(\mathrm{SC}_p)$] $E(\Q)[p] \neq 0$ and $E$ has more than one $\Q$-isogeny of degree $p$. 
	\item[$(\mathrm{B}'_p)$] $E(\Q)[p] \neq 0$ and $E$ has only one $\Q$-isogeny of degree $p$.
	\item[$(\mathrm{B}_p)$] $E(\Q)[p] = 0$ and there is a $\Q$-isogeny $E'\to E$ of degree $p$ with $E'(\Q)[p]\neq 0$.
\end{enumerate}
These conditions require that 
$E$ or an elliptic curve $E'$ isogenous to $E$, has a non-trivial $\Q$-rational $p$-torsion point 
for some odd prime $p$. 
Mazur's theorem on 
the torsion subgroup of the Mordell-Weil group $E(\Q)$  (\cite[Thm.~2]{Maz78}, \Cf\cite[Thm.~7.5]{Sil106})
says that 
the prime $p$ must be $3,5$ or $7$. 
For a semi-stable elliptic curve $E$ over $\Q$, an equality  
\[
\dimFp(E[p]_{G_{\Q}}) = 
\begin{cases}
1, &\mbox{if $(\mathrm{SC}_p)$ or $(\mathrm{B}_p)$ holds},\\
0, &\mbox{otherwise}
\end{cases}
\]
holds (\autoref{lem:quot}). 
For the even prime $p=2$, 
$\dim_{\F_2}(E[2]_{G_{\Q}}) \neq 0$ if and only if $E(\Q)[2]\neq 0$ (\autoref{lem:even}). 
By SageMath \cite{SAGE}, 
there are 21027 semi-stable elliptic curves $E$ over $\Q$ with conductor $<10000$. 
Within these, 
12201 curves satisfy $E[p]_{G_{\Q}}\neq 0$ for some prime $p$. 
\begin{ex}\label{ex:651}
We consider an isogeny class of elliptic curves over $\Q$ with conductor $651$. 
In this class, 
there are 3 semi-stable elliptic curves $E^{(1)}, E^{(2)}$ and  $E^{(3)}$ of the Cremona label 651e1, 651e2 and 651e3 respectively 
(\Cf\cite[\href{https://beta.lmfdb.org/EllipticCurve/Q/651/b/}{Elliptic Curve 651.b}]{lmfdb}). 
There are isogenies 
%\[
%E^{(3)} \leftarrow  E^{(2)}  \rightarrow E^{(1)}
%\]
\[
\xymatrix@R=0mm{
E^{(3)} & E^{(2)} \ar[l]\ar[r]& E^{(1)} \\
\text{651e3} & \text{651e2} & \text{651e1}
}
\]
of degree $3$. 
Their Mordell-Weil groups are  
$E^{(1)}(\Q)[3] \simeq E^{(2)}(\Q)[3] \simeq \Z/3\Z$ and $E^{(3)}(\Q)[3] = 0$. 
The curve $E^{(2)}$ satisfies ($\mathrm{SC}_3$) and has 
split multiplicative reduction at $3$, $7$ and $31$. 
By computing the Hilbert symbol map 
(see \autoref{lem:multQ}, and \autoref{rem:tame}), 
we have 
\begin{align*}
\dim_{\F_3}\left(V(E^{(2)}_3)/3V(E^{(2)}_3)\right) &= 0,\ \mbox{and}\\
\dim_{\F_3}\left(V(E^{(2)}_7)/3V(E^{(2)}_7)\right) &= \dim_{\F_3}\left(V(E^{(2)}_{31})/3V(E^{(2)}_{31})\right) = 1, 
\end{align*}
where $E^{(2)}_l := E^{(2)}\otimes_{\Q}\Q_l$. 
By \autoref{lem:surj}, 
$\Coker(\dbar_{E,p}) = 0$ and hence 
\autoref{thm:introEQ} says 
%$V(E^{(2)})$ is $l$-divisible for all prime $l\neq 3$ and 
$\dim_{\F_3}(\Ker(\dbar_{E^{(2)},3})) = 1$. 
The boundary map $\dbar_{E^{(2)},3}$ induces an exact sequence
\[
0\to \Z/3\Z \to  V(E^{(2)})/3V(E^{(2)}) \xrightarrow{\dbar_{E^{(2)},3}}\bigoplus_{l\colon\mathrm{good}}{\ol{E^{(2)}_l}(\F_l)}/{3\ol{E^{(2)}_l}(\F_l)} \to 0.
\] 
The curve $E^{(3)}$ satisfies ($\mathrm{B}_3$) and 
has also split multiplicative reduction at $3, 7$ and $31$.  
By \autoref{lem:multQ},  
\[
V(E^{(3)}_3)/3V(E^{(3)}_3) = V(E^{(3)}_7)/3V(E^{(3)}_7) = V(E^{(3)}_{31})/3V(E^{(3)}_{31}) = 0. 
\]
\autoref{thm:introEQ} gives an exact sequence 
\[
0\to  V(E^{(3)})/3V(E^{(3)}) \xrightarrow{\dbar_{E^{(3)},3}} \bigoplus_{l\colon\mathrm{good}}\ol{E^{(3)}_l}(\F_l)/3\ol{E^{(3)}_l}(\F_l) \to \Z/3\Z \to 0.
\] 
Finally, as $E^{(1)}$ satisfies $(\mathrm{B}'_3)$, we have $E^{(1)}[3]_{G_\Q} = 0$ (\autoref{lem:quot}). 
\end{ex}

%By SageMath \cite{SAGE}, 
%there are 9913 semi-stable elliptic curves $E$ over $\Q$ with conductor $<10000$ under the condition $E(\Q)[2] = 0$. 
%Within these, 
%the 950 elliptic curves satisfy the condition $(\mathrm{SC}_p)$ for some $p\in \set{3,5,7}$ 
%and 950 curves satisfy $(\mathrm{B}_p)$. 
%In both of these cases, 
%$V(E)$ is $l$-divisible for all $l\neq p$ 
%and 
%the kernel and the cokernel of the boundary map 
%$\dbar_{E,p}$ may not be trivial by \autoref{thm:introEQ}. 
%For all the remaining 8013 curves, $V(E)$ becomes divisible. 
For the case where $E(\Q)[2]\neq 0$ or $E$ has non-split multiplicative reduction at some prime, 
our approach only provides upper bounds of 
$\dimFp(\Ker(\dbar_{E,p}))$ and $\dim_{\Fp}(\Coker(\dbar_{E,p}))$ (see \autoref{ex:35a1}, \autoref{ex:17a2}).

\subsection*{Notation}
For a field $F$, 
let $L/F$ be a Galois extension with $G = \Gal(L/F)$, 
and $M$ a $G$-module. 
For each $i\in\Z_{\ge 0}$, we denote by 
$H^i(L/F, M) = H^i_{\mathrm{cont}}(G,M)$ the $i$-th continuous Galois cohomology group. 
If $L$ is a separable closure of $F$, then we write $H^i(F,M) = H^i(L/F,M)$. 
For an elliptic curve $E$ over a field $F$ and a field extension $L/F$, 
we denote by $E_L:=E\otimes_F L$ the base change to $L$. 

By a \textbf{number field}, we mean a finite field extension of the rational number field $\Q$. 
For a number field $F$, we use the following notation: 
\begin{itemize}
	\item $P(F)$: the set of places in $F$, 
	\item $P_{\mathrm{fin}}(F)$: the subset of $P(F)$ consisting of finite places, 
	\item $P_{\infty}(F) := P(F)\ssm P_{\fin}(F)$: the set of infinite places in $F$, and 
	\item $G_F := \Gal(\ol{F}/F)$ the absolute Galois group of $F$.   
\end{itemize}
For each place $v\in P(F)$, define 
\begin{itemize}
	\item $F_v$: the completion of $F$ at $v$, 
	\item $v\colon F_v^{\times}\to \Z$: the valuation map of $F_v$, 
	\item $\O_{F_v}$: the valuation ring of $F_v$, and 
	\item $\F_v := \O_{F_v}/\m_v$: the residue field of $F_v$.
\end{itemize}
For an abelian group $G$ and $m\in \Z_{\ge 1}$,  
we write $G[m]$ and $G/m$ for the kernel and cokernel 
of the multiplication by $m$ on $G$ respectively. 
%We also denote by $G\{m\} := \bigcup_{n\ge 1}G[m^n]$ the $m$-primary part of $G$. 
%For a profinite group $G$, 
%and a $G$-module $M$, we denote by $M^G \subset M$ and $M\surj M_G$ 
%its $G$-invariant subgroup and $G$-coinvariant quotient, respectively.
%In this note, by a \textbf{variety} over $k$ we mean an integral and separated scheme of finite type over $k$,  
%and a \textbf{curve} over $k$ is a variety over $k$ with dimension $1$. 

A \textbf{curve} over a field $F$ we mean an integral scheme of dimension 1, of finite type over $F$. 

\subsection*{Acknowledgements} 
The author thanks Prof.~Yoshiyasu Ozeki for his comments on the mod $p$ Galois representations in this note.
The author was supported by JSPS KAKENHI Grant Number 24K06672.

%%% 
\section{Class field theory}\label{sec:cft}

\subsection*{Abelian fundamental groups for curves}
Let $F$ be a field of characteristic 0, and 
$X$ a projective smooth curve over a field $F$ with $X(F)\neq \emptyset$. 
Note that the assumption $X(F)\neq \emptyset$ implies $X$ is geometrically connected. 
% As $F$ is perfect and $X$ is reduced, $X$ is geometrically reduced (\cite[Chap.~3, Ex.~2.10]{Liu}). 
% Let $X_{\Fbar} = Y_1 \cup \cdots \cup Y_n$ be the decomposition to the connected components. 
% The Galois group $G_F$ acts on $X_{\Fbar} = X\otimes_F \ol{F}$ via the second factor. 
%This action permutes the set of connected components $\set{Y_1,\ldots ,Y_n}$ 
%and acts transitively.
%A rational point $x_0 \in X(F)$ gives rise $\ol{x}_0 \in X_{\Fbar}(\Fbar)$ 
%and the point $\ol{x}_0$ is fixed by $G_F$. 
%Therefore, $n=1$ and $X_{\Fbar}$ is connected. 
  We denote by $X_0$ the set of closed points in $X$. 
The group $SK_1(X)$ is defined by the cokernel of the tame symbol map 
\[
SK_1(X) = \Coker\left(\dt_{F(X)} \colon K_2^M(F(X)) \to \bigoplus_{x\in X_0} F(x)^{\times}\right), 
\]
where  
$F(x)$ is the residue field at $x \in X_0$, and 
$F(X)$ is the function field of $X$. 
The norm maps $N_{F(x)/F}\colon F(x)^{\times} \to F^{\times}$ 
for closed points $x \in X_0$ 
induce $N\colon SK_1(X) \to F^{\times}$. 
Its kernel is denoted by $\VX$. 
From the assumption $X(F)\neq \emptyset$, 
the map $N$ is surjective and the short exact sequence 
\[
0\to V(X) \to SK_1(X)\to F^{\times} \to 0
\]
splits. 
The Milnor type $K$-group $K(F;J,\Gm)$ 
associated to the Jacobian variety $J := \Jac_X$ of $X$ and the multiplicative group $\Gm$ 
is generated by symbols $\set{P,f}_{F'/F}$ 
of $P\in J(F')$ and $f\in \Gm(F') = (F')^{\times}$ 
for a finite field extension $F'/F$ 
(for the definition of the Somekawa $K$-group, see \cite{Som90}, \cite{RS00})
By \cite{Som90}, 
there is a canonical isomorphism 
\begin{equation}
\label{eq:Som}	
\varphi\colon V(X)\isomto K(F;J,\Gm)
\end{equation}
after fixing $x_0 \in X(F)$. 
For each $x\in X_0$ and $f \in (F(x))^{\times}$, 
the map $\varphi$ is given by 
\[
	\varphi(f) = \set{[x] - [x_0],f}_{F(x)/F}.
\]
On the other hand, 
there is a split exact sequence
\[
0 \to \piabXgeo\to \piabX\to G_F^{\ab}\to 0
\]
of abelian fundamental groups, where
$G_F^{\ab} = \Gal(F^{\ab}/F)$ is the Galois group of the maximal abelian extension $F^{\ab}$ of $F$, and 
$\piabXgeo$ is defined by the exactness. 
It is known that the geometric part 
$\piabXgeo$ is isomorphic to 
the $G_F$-coinvariant quotient $T(X)_{G_F}$ of the full Tate module 
$T(X) = \prod_{l\colon \text{prime}} T_l(X)$, where  
$T_l(X) := \plim_n J[l^n]$ 
and $J[l^n] := J(\ol{F})[l^n]$ is the group of $l^n$-torsion points of $J(\ol{F})$ (\Cf\cite{KL81} and \cite[Sect.~3]{KS83b}).

For any prime number $p$, 
it is known that the \textbf{Galois symbol map} 
\begin{equation}
	\label{eq:sFp}
	s_{F,p}\colon V(X)/p \simeq K(F;J,\Gm)/p \hookrightarrow H^2(F,J[p](1)) = H^2(F,J[p]\otimes \mu_p)
\end{equation}
is injective, where $\mu_p$ is the group of $p$-th roots of unity (\cite[Thm.~6.1]{Yam05}).

\subsection*{Class field theory for curves over a \texorpdfstring{$p$}{p}-adic field}
Let $K$ be a finite field extension of $\Qp$ 
and $X_K$ be a projective smooth and geometrically irreducible curve over $K$. 
Following \cite{Blo81}, \cite{Sai85a} and \cite{KS83b}, 
we recall the class field theory for the curve $X_K$. 
%Recalling from Notation, $P_{\fin}(F)$ is the set of finite places in $F$. 
%For each  $v\in P_{\fin}(F)$, 
%put $X_v := X\otimes_F F_v$, where $F_v$ is the local field associated to $v$
A map 
\[
\sigma_{X_K}\colon SK_1(X_K)\to \piab(X_K) 
\]
called the \textbf{reciprocity map} makes the following diagram commutative:
\[
\xymatrix{
0 \ar[r] & V(X_K)\ar[r]\ar[d]^{\tau_{X_K}} &SK_1(X_K)\ar[r]^-{N}\ar[d]^{\sigma_{X_K}} &K^{\times} \ar[r]\ar[d]^{\rho_{K}} & 0 \\
0 \ar[r] & \piab(X_K)^{\geo}\ar[r] & \piab(X_K)\ar[r] & G_{K}^{\ab}\ar[r] & 0,
}
\]
where $\rho_{K}$ is the reciprocity map of local class field theory.
\begin{thm}[\cite{Blo81},\cite{Sai85a}]
\label{thm:BS}
Let $X_K$ be a projective smooth and geometrically irreducible curve over $K$. 

\begin{enumerate}
	\item The kernel $\Ker(\sigma_{X_K})$ $($resp.~$\Ker(\tau_{X_K})$$)$ is the maximal divisible subgroup of $SK_1(X_K)$ $($resp.~$V(X_K)$$)$.
	\item The image $\Im(\tau_{X_K})$ is finite. 
	\item The cokernel $\Coker(\tau_{X_K})$ and the quotient $\piab(X_K)/\ol{\Im(\sigma_{X_K})}$ 
	of $\piab(X_K)$ by the topological closure $\ol{\Im(\sigma_{X_K})}$ of the image 
	of $\sigma_{X_K}$ is isomorphic to $\Zhat^r$ for some $r\ge 0$. 
\end{enumerate}	
\end{thm}

There is a proper flat scheme $\X_{\O_K}$ over $\O_{K}$ of $X_K$ such that 
the generic fiber is $\X_{\O_K} \otimes_{\O_{K}}K = X_K$. 
The special fiber $\X_{\O_K} \otimes_{\O_{K}}\F_K$ is denoted by 
$\XKbar$, where $\F_K$ is the residue field of $K$. 
Recall that $X_K$ is said to have \textbf{good reduction} 
if the special fiber $\XKbar$ is also smooth over the finite field $\F_K$. 
Now, we assume $X_K$ has good reduction and $X_K(K)\neq \emptyset$. 
By \cite[Sect.~2, Cor.~1]{KS83b}, 
the boundary map 
\[
 \bigoplus_{x \in (X_K)_0\subset (\X_{\O_K})_1}K_1(K(x)) \to \bigoplus_{\ol{x}\in (\XKbar)_0 = (\X_{\O_K})_0}K_0(\F_K(\ol{x}))
\]
of the $K$-groups 
(which is given by the valuation map $K(x)^\times \to \Z$)
induces a map 
\begin{equation}
\label{eq:dv}	
\d_{X_K}\colon SK_1(X_K)\to CH_0(\XKbar)
\end{equation}
which is surjective. 
There is a commutative diagram with exact rows
\begin{equation}\label{diag:d_X}
\vcenter{
\xymatrix{
	0 \ar[r] & V(X_K)\ar[d]^{\d_{X_K}} \ar[r] & SK_1(X_K)\ar[d]^{\d_{X_K}}  \ar[r]^-{N} & K^{\times}\ar[d]^{v_K}\ar[r] & 0\\
	0 \ar[r] & A_0(\XKbar) \ar[r] & CH_0(\XKbar)	\ar[r]^-{\deg} & \Z \ar[r] & 0, 
}}
\end{equation}
where the right vertical map $v_K$ is the valuation map of $K^{\times}$. 
The above diagram induces the local boundary map 
\begin{equation}\label{eq:locald}
	\d_{X_K}\colon V(X_K)\to A_0(\XKbar)\simeq \Jac_{\XKbar}(\F_K)\simeq \JKbar(\F_K), 
\end{equation}
where $\Jac_{\XKbar}$ is the Jacobian variety of the variety $\XKbar$ 
and $\JKbar$ is the reduction of the Jacobian variety $J_K = \Jac_{X_K}$ of $X_K$. 
Since the horizontal maps in \eqref{diag:d_X} split, 
the map $\d_{X_K} \colon V(X_K)\to \JKbar(\F_K)$ is also surjective. 
Precisely, fixing $x_0 \in X_K(K)$ and identifying the isomorphism 
$V(X_K)\simeq K(K;J_K,\Gm)$, 
for a finite extension $L/K$, $P\in J(L)$ and $f\in L^{\times}$, 
the map $\d_{X_K}$ is given by 
\begin{equation}
\d_{X_K}(\set{P,f}_{L/K}) = v_{L}(f)N_{\F_L/\F_K}(\Pbar),  
%f\, \mbox{at $x$} - f\,\mbox{at $x_0$}	) =\mathrm{ord}_{F_v(x)}(f)([\xbar] - [\xbar_0]), 
\end{equation}
where $v_L$ is the valuation map of the local field $L$, 
$\Pbar$ is the image of $P$ by the reduction map $\red_L\colon J_L(L)\to \ol{J_L}(\F_L)$, 
and $N_{\F_L/\F_K}\colon \Jbar_L(\F_L)\to \Jbar_K(\FK)$ is the norm map.
%(cf. \cite[(4.12)]{Blo81})

%There is a 
%short exact sequence 
%$0\to \piab(X_K)^{\geo}\to \piab(X_K) \to G_{K}^{\ab}\to 0$ as above. 
%Moreover, we have 
There is a surjective map $\sp_{X_K}\colon \piab(X_K)^{\geo}\to \piab(\XKbar)^{\geo}$ 
and its kernel is denoted by $\piab(X_K)^{\geo}_{\ram}$ (\Cf\cite{Yos02}).
The classical class field theory (for the curve $\XKbar$ over $\F_K$)   
says that the reciprocity map $\rho_{\XKbar}:A_0(\XKbar) \isomto \piab(\XKbar)^{\geo}$ is bijective 
of finite groups %(\cite[Section 1]{KS83b})
and makes 
the following diagram commutative: 
\begin{equation}
\label{diag:sp}	
\vcenter{\xymatrix{
	0 \ar[r] & \Ker(\d_{X_K})\ar@{-->}[d]^{\mu_{X_K}}  \ar[r] & V(X_K)\ar[r]^-{\d_{X_K}}\ar@{->>}[d]^{\tau_{X_K}}   &  A_0(\XKbar)\ar[d]^{\rho_{\XKbar}}_{\simeq} \ar[r] & 0\,\\ 
	0 \ar[r] & \piab(X_K)^{\geo}_{\ram} \ar[r] &\piab(X_K)^{\geo}\ar[r]^-{\sp}  &  \piab(\XKbar)^{\geo} \ar[r] & 0.
	}
	}
\end{equation}
For the commutativity of the right square in the above diagram, 
see \cite[Prop.~2]{KS83b}. 
The reciprocity map $\mu_{X_K}$ induces an isomorphism of finite groups
\begin{equation}\label{eq:mu}	
	\Ker(\d_{X_K})/\Ker(\d_{X_K})_{\div} \isomto \piab(X_K)^{\geo}_{\ram},
\end{equation}
where $\Ker(\d_{X_K})_{\div}$ is the maximal divisible subgroup of $\Ker(\d_{X_K})$
(\Cf\cite[Sect.~2]{GH21}). 

\subsection*{The exact sequence of Bloch}
In the following, 
we assume that $F$ is a \textbf{number field}, that is, a finite extension of $\Q$ (cf.\ Notation). 
Let $X$ be a projective smooth curve over $F$ with $X(F)\neq \emptyset$.
For each $v\in P(F)$, we denote by $X_v$ the base change $X\otimes_F F_v$ of 
$X$ to the local field $F_v$. 
Put
\begin{align*}
	\Sigma_\good(X) &:= \set{v \in P_{\fin}(F) | \mbox{$X$ has good reduction at $v$}}, \mbox{and}\\
\Sigma_\bad(X) &:= P_{\fin}(F)\ssm \Sigma_\good(X).
\end{align*}
For the curve $X$, 
we denote by $V(X)_{\tor}$ the torsion subgroup of $V(X)$. 
As noted in Introduction,
Bloch's conjecture says the equality   
$V(X) = V(X)_{\tor}$ holds (\cite[Rem.~1.24]{Blo81}).

\begin{lem}
\label{lem:VXtor}
For a prime $p$, the inclusion map $V(X)_{\tor} \hookrightarrow V(X)$ 
gives an isomorphism $V(X)_{\tor}/p \xrightarrow{\simeq} V(X)/p$. 
\end{lem}
\begin{proof}
As evidence for Bloch's conjecture on $V(X)$, Raskind proved that 
for a base change $X_{\Fbar} = X\otimes_F \Fbar$, 
there is a short exact sequence 
\[
0 \to V(X)_{\tor} \to V(X) \to V(X_{\Fbar})^{G_F}\to 0
\]
where the fixed subgroup $V(X_{\Fbar})^{G_F}$ is uniquely divisible (\cite[Thm.~0.2]{Ras90}). 
The above exact sequence induces 
\[
(V(X_{\Fbar})^{G_F})[p] \to  V(X)_{\tor}/p \to V(X)/p \to V(X_{\Fbar})^{G_F}/p \to 0.
\]
As $V(X_{\Fbar})^{G_F}$ is uniquely divisible so that torsion free, 
we obtain 
\[
(V(X_{\Fbar})^{G_F})[p] = V(X_{\Fbar})^{G_F}/p = 0.
\]
The assertion follows from these equalities.
\end{proof}

\begin{prop}[{\cite[Sect.~5, Prop.~5]{KS83b}}]
\label{prop:KS}
Let $X$ be a projective smooth curve over a number field $F$ with $X(F)\neq \emptyset$.
%Let $X$ be a proper smooth and geometrically irreducible curve over a number field $F$. 
\begin{enumerate}
	\item $T(X)_{G_F} \simeq \piabXgeo$ is finite 
	and $T(X)_{G_{F_v}} \simeq T(X_v)_{G_{F_v}}\simeq \piab(X_v)^{\geo}$ are finite for almost all places $v\in P(F)$.
	\item %Let $m$ be a non-zero integer which annihilates the finite group $T(X)_{G_F}$ 
	Put $m_X = \# (T(X)_{G_F})$. % and we assume $X(F)\neq \emptyset$. 
	Then, we have an exact sequence 
	\[
	V(X) \xrightarrow{\loc} \bigoplus_{v\in P(F)} V(X_{v})/m_X \to (T(X))_{G_F}\to 0.
	\]
\end{enumerate}
\end{prop}

By composing the local boundary map \eqref{eq:locald}, we obtain the global boundary map  
\begin{equation}\label{eq:dX}
	\d_X\colon V(X)\xrightarrow{\loc} \prod_{v\in P(F)} V(X_v)\xrightarrow{\prod \d_{X_v}} \prod_{v\in \Sigma_\good(X)}\Jvbar(\F_v).
\end{equation} 
By the proof of \cite[Sect.~5, Prop.~5]{KS83b}, 
	the image of 
	\[
	V(X) \to \prod_{v\in P(F)}V(X_v)\xrightarrow{\prod \tau_{X_v}} \prod_{v\in P(F)}(T(X_v))_{G_{F_v}}
	\]
	is contained in the direct sum $\bigoplus_{v}T(X_v)_{G_{F_v}}$.
	Since the boundary map $\d_{X_v}$ factors through $\tau_{X_{v}}$ (cf.~\eqref{diag:sp}), 
%	The image of $\dX$ is contained in the direct sum. 
	the image of $\dX$ is contained in the direct sum $\bigoplus_{v\in \Sigma_{\good}(E)}\Jvbar(\Fv)$.

\section{Elliptic curve}\label{sec:EC}
%Let $F$ be a number field. 
%In the following, 
%we assume that $F$ is a number field, that is, a finite extension of $\Q$ (cf.\ Notation). 
%%Let $X$ be a projective smooth curve over $F$ 
%%with $X(F)\neq \emptyset$.
%Put
%\begin{align*}
%	\Sigma_\good(X) &:= \set{v \in P_{\fin}(F) | \mbox{$X$ has good reduction at $v$}}, \mbox{and}\\
%\Sigma_\bad(X) &:= P_{\fin}(F)\ssm \Sigma_\good(X).
%\end{align*}
%There is a short exact sequence 
%$0\to \piab(X_v)^{\geo}\to \piabX \to G_{F_v}^{\ab}\to 0$ as above. 
%Moreover, we have a surjective map $\sp_{X_v}\colon \piab(X_v)^{\geo}\to \piab(\Xvbar)^{\geo}$ 
%and its kernel is denoted by $\piab(X_v)^{\geo}_{\ram}$ (\Cf\cite{Yos02}).
Let $E$ be an elliptic curve over a number field $F$. 
For any place $v\in P(F)$, we denote by $F_v$ the local field associated to $v$ (\Cf Notation) 
and put $E_v := E\otimes_F F_v$. 

\subsection*{A Hasse principle}
%The target of the map $s_{F,p}$ is isomorphic to the $G_F$-coinvariant: 
%$H^2(F,E[p](1))\simeq E[p]_{G_F}$ (\cite[(2.21)]{Blo81}, see also \cite[Prop.~A.1]{GH23}).
%From this reason, we obtain the following lemma.
For a rational prime $p$, we consider the natural map 
\begin{equation}\label{def:loc}
	\locpbar \colon V(E)/p \to \prod_{v\in P(F)} V(E_v)/p.
\end{equation}

\begin{lem}\label{lem:Ram}
	The map $\locpbar$ in \eqref{def:loc} is injective, 
	and the image $\Im(\locpbar)$ is contained in 
	$\bigoplus_{v\in P(F)}V(E_v)/p$.
\end{lem}
\begin{proof} 
	By using the Somekawa $K$-group associated to $E$ and $\Gm$, 
	there is an isomorphism $V(E) \simeq K(F;E,\Gm)$ (\Cf\eqref{eq:Som}). 
	For any prime number $p$, the Galois symbol map 
	\begin{equation}\label{eq:Galsymb}
		s_{F,p}\colon V(E)/p\to H^2(F, E[p](1))
	\end{equation}
	and the local Galois symbol map 
	\[
	s_{F_v,p}\colon V(E_v)/p \hookrightarrow H^2(F_v,E_v[p](1))
	\]
	for $v\in P(F)$ are injective (\Cf\eqref{eq:sFp}). 
	There is a commutative diagram below:
	\begin{equation}\label{diag:s_F}
	\vcenter{
	\xymatrix{
	V(E)/p \ar[r]^-{\locpbar}\ar@{^{(}->}[d]_{s_{F,p}} &\displaystyle\prod_{v\in P(F)} V(E_v)/p \ar@{^{(}->}[d]^-{s_{F_v,p}} \\
	H^2(F, E[p](1)) \ar[r]^-{\loc^2_p} & \displaystyle\prod_{v\in P(F)} H^2(F_v, E_v[p](1)).
	}}
	\end{equation}
	Here, the bottom horizontal map $\loc^2_p$ is given by the restriction maps 
	on the Galois cohomology groups. 
	By the commutative diagram above, 
	the assertion is reduced to showing $\loc^2_p$ is injective. 
	By the Tate global duality theorem, 
	the kernel (which is denoted by $\Sha^2(F,E[p](1))$ in \cite{MilADT}) of the bottom horizontal map $\locpbar^2$ in the diagram above
	is the Pontrjagin dual of the kernel of 
	\[
	\loc^1_p\colon H^1(F, E[p](1)^D)\to \prod_{v\in P(F)} H^1(F_v, E_v[p](1)^D),
	\]
	where $E[p](1)^D := \Hom(E[p](1), (F^{\sep})^{\times}) \simeq \Hom(E[p],\Z/p) = E[p]^{\vee}$(\cite[Chap.~I, Thm.~4.10]{MilADT}, \cite[Chap.~VIII, Thm.~8.6.7]{NSW08}). 
	For the extension $K := F(E[p])$ of $F$, 
	the inf-res exact sequence (\cite[Chap.~I, Prop.~1.6.7]{NSW08}) gives a commutative diagram with left exact horizontal sequences:
	\[
	\xymatrix@C=4mm{
	H^1(K/F, (E[p]^{\vee})^{G_{K}})\ar@{^{(}->}[r]\ar[d]^{\loc_{K/F}^1} & H^1(F,E[p]^{\vee})\ar[r]\ar[d]^{\loc^1_p} & H^1(K,E[p]^{\vee})\ar[d]^{\loc_{K}^1}\\
	\displaystyle\prod_{v\in P(F)}\prod_{w\mid v} H^1( K_w/F_v, (E_v[p]^{\vee})^{G_{K_w}})\ar@{^{(}->}[r] &\displaystyle \prod_{v\in P(F)} H^1(F_v, E_v[p]^{\vee})\ar[r] &\displaystyle\prod_v\prod_{w\mid v}  H^1(K_w,E_v[p]^{\vee}),
	}
	\]
	where $w\mid v$ means that $w$ runs through the set of places of $K$ above $v \in P(F)$. 
	A basis $E[p] \subset E(K)$ as a $\Fp$-vector space 
	also provides a basis of $E_v[p]$. 
	Since $G_{K}$ acts on $E[p]^{\vee}$ trivially, we have 
	$H^1(K,E[p]^\vee) = \Hom(G_K,\Z/p)^{\oplus 2}$ 
	and $H^1(K_w,E_v[p]^{\vee}) = \Hom(G_{w},\Z/p)^{\oplus 2}$. 
	Since the natural map 
	\[
		\Hom(G_K,\Z/p)\to \prod_{w\in P(K)}\Hom(G_w,\Z/p)
	\]
	is injective (by \cite[Chap.~VI, Cor.~3.8]{Neu99}), 
	the right vertical map $\loc_K^1$ in the above diagram is injective. 
	% For any $\chi \in G_K^\vee$, assume $\chi_w: G_w \to \Z/p$ is zero and $\chi \neq 0$. 
	% AS \chi \neq 0, then the extension L/K corresponding to the kernel of chi 
	% is an extension of degree p. by the assumption, the decomposition group at any w is trivial 
	% so that the extension L/K is completely split at all w. By Cor. 3.8 in Neukirch, we have L=K.
	
	Finally, we show that the left vertical map $\loc_{K/F}^1$ is injective. 
	Since $G_{K}$ acts on $E[p]^{\vee}$ trivially, 
	$(E[p]^{\vee})^{G_K} = E[p]^{\vee} =:M$.
	For each $v\in P(F)$ and $w\mid v$, by fixing the embeddings 
	\[
	\xymatrix{
	\ol{F} \ar@{^{(}->}[r] &  \ol{F}_v\, \\ 
	F\ar[r]\ar@{^{(}->}[u]\ar@{^{(}->}[r] & F_v,\ar@{^{(}->}[u]
	}
	\]
	there is an isomorphism 	$E[p] \simeq E_v[p]$ as $G_{F_v}$-modules. 
	Using the identification of $G_{F_v}$ and the decomposition subgroup at $v$ of $G_F$, 
	$(E_v[p]^\vee)^{G_{K_w}} \simeq (E[p]^{\vee})^{G_{K_w}} = M$ 
	as $\Gal(K_w/F_v)$-modules. 
	Moreover, the Galois group $\Gal(K_w/F_v)$ can be regarded as a cyclic subgroup 
	of $\Gal(K/F)$ if $w$ is an unramified place. 
	The extension $K/F$ corresponding to the kernel of the mod $p$ Galois 
	representation $\rho_{E,p}\colon G_F\to \Aut(E[p]) \simeq GL_2(\Fp)$. 
	We consider $G := \Gal(K/F)$ as a subgroup of $GL_2(\Fp)$ after fixing a basis of $E[p]$. 
	It is enough to show that the natural map 
	\[
	\loc_{G}\colon H^1(G, M) \to \prod_{D \subset G} H^1(D, M) 
	\]
	is injective, 
	where $D$ runs through the set of decomposition groups $\Gal(K_w/F_v)$ of $G$. 
	By the Chebotarev density theorem (\cite[Chap.~VII, Thm.~13.4]{Neu99}), 
	for any cyclic subgroup $C\subset G$, 
	one can find unramified place $w \in P(K)$ such that 
	$C$ is isomorphic to $\Gal(K_w/F_v)$. 
%	Let $\sigma\in C$ be a generator and $c = \set{\tau\sigma \tau^{-1} | \tau \in G}$ 
%	the conjugacy class in $G$. 
%	Put $P(c) := \set{v\in P_{\fin}(F);\mbox{unramified} | \left(\frac{K/F}{\m_v}\right)= c}$ 
%	has Dirichlet density $\dfrac{\#c}{\#G}>0$. 
%	Take $v\in P(c)$ and a place $w\mid v$ of $K$. 
%	$\Gal(K_w/F_v)$ is cyclic generated by the Frobenius $\Frob_w$.
%	Since the Artin symbol $\left(\frac{K/F}{\m_v}\right)$ is defined as the conjugacy class of 
%	the Frobenius $\Frob_w$ in $G$, 
%	we have $C \simeq \braket{\Frob_w} = \Gal(K_w/F_v)$. 
	Now, we apply the Hasse principle for a subgroup of $GL_2(\Fp)$ (see \autoref{prop:Ram} below) 
	to deduce that the map $\loc_G$ is injective.  	
	% In fact, for M = (E[p]^\vee)^G_K
%	\[
%	H^1(K/F,M) \to \prod_{w\in P(K)\colon \text{unramified}}H^1(K_w/F_v,M)
%	\]
%	is injective and this implies the assertion.
	This implies that $\loc^1_p$ is injective by the five lemma  
	and so is $\locpbar$.
	
	For the image $\Im(\locpbar)$, the image is contained in $\bigoplus_{v\in P(F)} V(E_v)/p$ 
	because of the commutative diagram \eqref{diag:s_F} 
	and the image of $\loc_p^2\colon H^2(F,E[p](1)) \to \prod_v H^2(F_v,E_v[p](1))$ is contained in 
	the direct sum $\bigoplus_v H^2(F_v,E_v[p](1))$ (\cite[Chap.~I, Lem.~4.8]{MilADT}). 
\end{proof}

\begin{prop}[{\cite[Prop.~1.2.1]{Ram25}}]\label{prop:Ram}
	Let $G$ be a subgroup of $GL_2(\Fp)$. Then, for any $p$-primary $G$-module $M$, 
	the natural map 
	\[
	H^1(G,M) \to \prod_{C\subset G} H^1(C,M)
	\]
	is injective, 
	where $C$ runs the set of cyclic subgroups of $G$.
\end{prop}

%\begin{align*}
%	(T_p(E)_{G_F})/p &= (\plim_n E[p^n])_{G_F}/p\\
%	&\stackrel{(\star)}{\simeq} (\plim_n E[p^n]_{G_F})/p \\
% & \stackrel{(\diamond)}{\simeq} \plim_n (E[p^n]_{G_F}/p)\\
% &\simeq E[p]_{G_F}.
%\end{align*}
%Here, the isomorphism ($\star$) follows from \cite[Prop.~A.1]{GH23}. 
%Since $T_p(E)_{G_F}$ is a finite $p$-group (\autoref{prop:KS}), 
%$E[p^n]_{G_F}$ is finite and its order is independed to $n$.  
%By the Mittag-Leffler condition for the exact sequence 
%$0 \to (pE[p^n]_{G_F})_n\to (E[p^n]_{G_F})_n \to (E[p^n]_{G_F}/p)_n\to 0$ 
%of inverse systems of abelian groups, the isomorphism ($\diamond$) holds.
%
%Put $A_n := E[p^n]_G$, and the translation map 
%$A_{n+1} \to A_n$ is given by the multiplication by $p$. 
%We want to show $(\plim A_n)/p  \simeq \plim (A_n/p)$. 
%The exact sequence $0\to pA_n\to A_n \to A_n/pA_n \to 0$ 
%induces 
%$0\to \plim pA_n \to \plim A_n \to \plim (A_n/pAn)$. 
%The last sequence is right exact when 
%($pA_n$) satisfies the Mittag-Leffler condition: 
%For every $m$, the sequence ($B_{n,m})_{n>m}$ is stationary. % (cf. \cite[Chap.~1, Exer.~3.15]{Liu}).   
%By $B_{n,m} = \mathrm{Im}(p E[p^n]_G \to p E[p^m]_G)$, the morphism of $p E[p^n]_G \to p E[p^m]_G$ is just multiplication by $p^{n-m}$. We fix $m$. Put $p^k = \#T_p(E)_G$. For $n = m+k+1 >m$, we have $p^{n-m} E[p^n]_G = p^{k+1}E[p^n]_G = 0$ because of $T_p(E)_G \twoheadrightarrow E[p^n]_G$.

\begin{thm}\label{thm:pdiv}
	If we have $E[p]_{G_F}\neq 0$, then 
	there is a short exact sequence
\begin{equation}
	\label{seq:Bp}
0 \to V(E)/p \xrightarrow{\locpbar}\bigoplus_{v\in P(F)} V(E_v)/p \to E[p]_{G_F} \to 0 
\end{equation}%	If we have $E[p]_{G_F} =  0$, then $V(E)$ is $p$-divisible. 
\end{thm}
\begin{proof}
By \autoref{prop:KS}, there is a right exact sequence 
\[
 V(E) \xrightarrow{\loc} \bigoplus_{v\in P(F)} V(E_v)/m_E \to T(E)_{G_F}\to 0, 
\]
where $m_E = \# (T(E)_{G_F})$. 
Applying $-\otimes_{\Z} \Z/p$, 
the sequence  
\begin{equation}
	\label{seq:Bp2}
V(E)/p \to \bigoplus_{v\in P(F)} V(E_v)/\gcd(m_E,p) \to (T(E)_{G_F})/p \to 0 
\end{equation}
is exact, where $\gcd(m_E,p)$ means the greatest common divisor of $m_E$ and $p$. 
$(T(E)_{G_F})/p \simeq (T_p(E)_{G_F})/p$ and 
the short exact sequence 
\[
T_p(E) \xrightarrow{p}T_p(E) \xrightarrow{\mathrm{projection}} E[p]\to 0
\]
induces 
\[
T_p(E)_{G_F} \xrightarrow{p}T_p(E)_{G_F}\to E[p]_{G_F}\to 0.
\]
Therefore, 
\begin{equation}
\label{eq:TG/p}
(T_p(E)_{G_F})/p\simeq E[p]_{G_F}.
\end{equation}
The assumption $E[p]_{G_F}\neq 0$ 
implies that the $p$-primary part of the finite group $T_p(E)_{G_F}$ is non-trivial. 
%Recall that $m_E$ is the order of the finite group $T_p(E)_{G_F}$ (\Cf \autoref{prop:KS}). 
We have $p\mid m_E$ and hence $\gcd(m_E,p) = p$. 
The exact sequence \eqref{seq:Bp} is left exact 
by \autoref{lem:Ram}. 
\end{proof}

\begin{rem}\label{rem:infty}
	We claim here that the $\Fp$-dimension of the target $\bigoplus_{v\in P(F)}V(E_v)/p$ 
	of the map $\locpbar$ is infinite. 
	Recall that 
	$\Sigma_{\mathrm{good}}(E)$ is the set of finite places $v$ of $F$ 
	such that $E$ has good reduction at $v$. 
	For a place $v\in \Sigma_{\good}(E)$ with $v\nmid p$, 
	the local boundary map $\d_{E_v,p}\colon V(E_v)/p\to \Evbar(\Fv)/p$ 
	is bijective (see the proof of \autoref{thm:fin} below) 
	and hence it is enough to show that the dimension of  
	\[
	\bigoplus_{v\in \Sigma_{\good}(E), v\nmid p}\Ebar(\Fv)/p
	\]
	is infinite. 
	The reduction map $\red_v\colon E_v[p] \xrightarrow{\simeq} \Evbar[p]$ is an isomorphism. 
	 %(cf. Bloch pf of prop. 2.4).
	By the exact sequence 
	\[
	0\to \Evbar(\Fv)[p]\to \Evbar(\Fv)\xrightarrow{p} \Evbar(\Fv) \to \Evbar(\Fv)/p\to 0,
	\] 
	the equality 
	\[
	\dimFp(\Evbar(\Fv)/p) = \dimFp(\Evbar(\Fv)[p])
	\] 
	holds. 
	The latter $\Evbar(\Fv)[p] = \Evbar[p]^{G_{\Fv}}$ coincides with the eigenspace for eigenvalue 1 of $\rho_{E,p}(\Frob_v)$, 
	where $\rho_{E,p}\colon G_F\to \Aut(E[p])$ is the mod $p$ Galois representation associated to $E[p]$ 
	and  $\Frob_v$ is a Frobenius element at $v$. 
	
	Here, first we assume that 
	$E[p]\not\subset E(F)$ and 
	$v$ is completely split in the (non-trivial) extension $F(E[p])/F$, 
	then $\rho_{E,p}(\Frob_v)$ is the identity in $GL_2(\Fp)$ 
	and hence $\dimFp(\Evbar(\Fv)[p]) = 2$. 
	By the Chebotarev density theorem (\cite[Chap.~VII, Thm.~13.4]{Neu99}), 
	there are infinitely many places $v$ which is completely split in the extension 
	$F(E[p])/F$  
	
	In the case where $E[p] \subset E(F)$, 
	(we have $\dimFp(E[p]_{G_F}) = \dimFp(E[p]) =2$, see \autoref{lem:quot}) 
	the representation $\rho_{E,p}$ is trivial so that $\rho_{E,p}(\Frob_v)$ is the identity 
	for any place $v \in \Sigma_{\good}(E)$. 
%	the reduction map $E(F)[p] \subset E(\Fv)[p] \xrightarrow{\red_v} \Evbar(\Fv)[p]$ 
%	is injective	(\cite[Chap.~VII, Prop.~3.1]{Sil106}), 
	We also have $\dimFp(\Evbar(\Fv)[p]) = \dimFp(\Evbar(\Fv)/p) = 2$ 

	As a result, 
%	For example, if we assume $E(F)[p]\neq 0$, 
%	then the reduction map $E(F)[p] \hookrightarrow \Evbar(\F_v)$ 
%	is injective (\cite[Chap.~VII, Prop.~3.1]{Sil106}) 
%	so that 
%	\[
%	\dimFp(\Evbar(\Fv)[p]) = \dimFp(\Evbar(\Fv)/p) \ge 1.
%	\] 
	\autoref{thm:pdiv} says $\dimFp(V(E)/p) = \infty$ if $E[p]_{G_F}\neq 0$. 
\end{rem}

For each prime $p$, the boundary map 
$\partial_E$ (defined in \eqref{eq:dX}) induces 
\begin{equation}\label{eq:tameE}
\dbar_{E,p} \colon V(E)/p \to \bigoplus_{v\in \Sigma_\good(E)}\ol{E_v}(\F_v)/p.
\end{equation}
For each good place $v\in \Sigma_\good(E)$, the local boundary map $\d_{E_v}$ 
for the base change $E_v$ gives  
\[
\dbar_{E_v,p}\colon V(E_v)/p \to \ol{E_v}(\Fv)/p.
\]

\begin{cor}\label{thm:fin}
Let $E$ be an elliptic curve over a number field $F$ and $p$ a rational prime. 
If we assume $E[p]_{G_F}\neq 0$, then  
there is an exact sequence 
	\begin{align*}	
0 \to \Ker(\dbar_{E,p})\to \bigoplus_{v\in \Sigma_\good(E),\, v\mid p}\Ker(\dbar_{E_v,p})&\oplus \bigoplus_{v \in \Sigma_\bad(E)}V(E_v)/p \oplus \bigoplus_{v\in P_{\infty}(F)\colon \mathrm{real}}V(E_v)/p  \\
	&\to E[p]_{G_F}\to \Coker(\dbar_{E,p})\to 0
	\end{align*}
of finite dimensional $\Fp$-vector spaces. 
Here, $v\mid p$ means that $v$ is a place of $F$ above $p$.
\end{cor}

\begin{proof}
From the assumption $E[p]_{G_F}\neq 0$ and $(T_p(E)_{G_F})/p \simeq E[p]_{G_F}$ (\Cf \eqref{eq:TG/p}), 
we have $p\mid m_E$, where $m_E := \#(T(E))_{G_F}$. 
%For a prime $p$ with $V(E)/p \neq 0$, 
The exact sequence \eqref{seq:Bp} and the local boundary map 
$\dbar_{E_v,p}$ induce
a commutative diagram: 
\begin{equation}\label{diag:Bloch}
	\vcenter{
\xymatrix@C=5mm{
0\ar[r] & V(E)/p\ar[r]^-{\locpbar}\ar[d]^{\ol{\partial}_E} & \ds \bigoplus_{v\in P(F)}V(E_v)/p \ar[d]^{\oplus \ol{\partial}_{E_v,p}} \ar[r] & E[p]_{G_F} \ar[r] & 0\\
& \ds \bigoplus_{v\in \Sigma_\good(E)} \Evbar(\F_v)/p \ar@{=}[r] & \ds \bigoplus_{v\in\Sigma_\good(E)} \Evbar(\F_v)/p, 
}}
\end{equation}
where the right vertical map is defined by $\dbar_{E_v,p}$ for each $v\in \Sigma_{\good}(E)$  
and the $0$-map for the other places. 
For each $v\in \Sigma_{\good}(E)$ with $v\nmid p$, 
the local boundary map  
$\dbar_{E_v,p}\colon V(E_v)/p \xrightarrow{\simeq} \Evbar(\F_v)/p$ 
is known to be bijective (\cite[Prop.~2.29]{Blo81}).
By comparing the kernels of the vertical maps in the diagram \eqref{diag:Bloch}, 
the map $\locpbar$ induces an injective homomorphism  
\begin{equation}
	\label{eq:Ker2}
\Ker(\dbar_{E,p})\ \stackrel{\locpbar}{\hookrightarrow} \bigoplus_{v\in \Sigma_{\good}(E),\, v \mid p}\Ker(\dbar_{E_v,p})\oplus \bigoplus_{v \in \Sigma_\bad(E)}V(E_v)/p\oplus\bigoplus_{v\in P_{\infty}(F)}V(E_v)/p.
\end{equation}
The number of the direct summand on the right is finite. 
For any finite place $v \in P_{\fin}(F)$, 
the reciprocity map $V(E_v)/p \hookrightarrow T(E_v)_{G_{F_v}}/p\simeq E_v[p]_{G_{F_v}}$ 
is injective (\autoref{thm:BS}). 
This indicates $\dimFp(V(E_v)/p) \le 2$. 
For an infinite place $v\in P_{\infty}(F)$, 
the Galois symbol map  
\[
s_v\colon V(E_v)/p \hookrightarrow H^2(F_v,E_v[p](1))
\]
is injective 
and the latter Galois cohomology group is finite. 
In particular, $V(E_v)/p = 0$ when $v$ is complex. 

Applying the snake lemma to the diagram \eqref{diag:Bloch}, we obtain the required long exact sequence.
\end{proof}

\subsection*{Local components}
In the following, 
we investigate each component of the second term in the exact sequence appearing in \autoref{thm:fin}: 
\begin{itemize}[itemsep=-5pt]
	\item $\Ker(\dbar_{E_v,p})$ for $v\in \Sigma_\good(E)$ with $v\mid p$ (\autoref{lem:Yos}),\\
	 \item $V(E_v)/p$ for a real $v \in P_{\infty}(F)$ (\autoref{lem:real}), and \\
	\item $V(E_v)/p$ for $v \in \Sigma_\bad(E)$ (\autoref{lem:mult}).
\end{itemize}
As noted in the proof of \autoref{thm:fin}, 
the inequality $\dimFp(V(E_v)/p)\le 2$ holds for each $v\in P_{\fin}(F)$.

\begin{lem}\label{lem:Yos}
	Let $v \in \Sigma_\good(E)$ with $v\mid p$. 
	Put $e_v = e(F_v/\Qp)$ the absolute ramification index of the local field $F_v$. 
	
	\begin{enumerate}
		\item If $e_v < p-1$, then $\Ker(\dbar_{E_v,p}) = 0$. 
		\item If $E$ has good ordinary reduction at $v$, then 
		$\dimFp(\Ker(\dbar_{E_v,p})) \le 1$.
%		If we further assume $E_v[p]\subset E_v(F_v)$, then 
%		$\dimFp(\Ker(\dbar_{E_v,p})) = 1$.
		\item If we assume $E_v[p] \subset E_v(F_v)$, then $\dimFp(\Ker(\dbar_{E_v,p}))= 2$.
		\end{enumerate} 
\end{lem}
\begin{proof}
	(i) 
	Recall that the local boundary map $\d_{E_v}\colon V(E_v)\to \ol{E_v}(\Fv)$ 
	is surjective. 
	It is known that the $p$-primary part of $\piab(E)^{\geo}_{\ram}$ is trivial 
	if $e_v<p-1$ (\cite[Thm.~4.1]{Yos02}). 
	By the class field theory for $E_v$ (\Cf\eqref{eq:mu}), 
	the reciprocity map induces $\Ker(\d_{E_v})/p \simeq \piab(E)^{\geo}_{\ram}/p$ 
	and hence $\Ker(\dbar_{E_v,p}) = 0$. 
	
	\noindent
	(ii) By \cite[Cor.~4.1]{GH23}, 
	there are surjective homomorphisms 
	\[
	\Z/p \twoheadrightarrow \Ker(\d_{E_v})/p \twoheadrightarrow \Ker(\dbar_{E_v,p}).
	\]
	The inequality $\dimFp(\Ker(\dbar_{E_v,p}))\le 1$ holds. 

	\noindent
	(iii) By \cite[Thm.~5.9]{GH23} and $\Ker(\d_{E_v})/p \simeq V(E_v)/p$, 
	we have  
	$\dim_{\Fp}(\Ker(\d_{E_v})/p) = 2$.
\end{proof}

\begin{lem}\label{lem:real}
	Let $v\in P_{\infty}(F)$ be a real place. 
	
	\begin{enumerate}
		\item 
	If $p>2$, then 
	$V(E_v)/p = 0$. 
	\item 
	For $p=2$, 
	we have 
	\[
	\dim_{\F_2}(V(E_v)/2) \le \begin{cases}
		1,& \mbox{if $\Delta(E_v)<0$},\\
		2,&  \mbox{if $\Delta(E_{v})>0$},
	\end{cases}
	\]
	where $\Delta(E_{v})$ is the discriminant of $E_{v}$.
	\end{enumerate}

\end{lem}
\begin{proof}
(i) 
The composition 
\[
V(E_\R)/p\xrightarrow{\res} V(E_\C)/p \xrightarrow{\Cor} V(E_\R)/p
\] 
is $[\C\colon \R] = 2$ and hence bijective. 
Since 
$V(E_\C)$ is uniquely divisible (\cite[Lem.~1.1]{Ras90}), 
we have $V(E_\R)/p = 0$. 

\noindent
(ii) 
The target of the Galois symbol map 
$V(E_{\R})/2\hookrightarrow H^2(\R,E_{\R}[2](1))$ 
is isomorphic to $H^1(\R,E_{\R}[2])$. 
The Tate cohomology group gives
$H^1(\C/\R,E_{\R}[2]) \simeq \wh{H}^1(\Gal(\C/\R), E_{\R}[2]) \subset E_{\R}[2]_{\Gal(\C/\R)}$ 
(\Cf\cite[Prop.~1.7.1]{NSW08}, \cite[Sect.~2]{Blo81}). 
The complex conjugation $\sigma \in \Gal(\C/\R)$ induces a short exact sequence 
\[
0 \to E_{\R}(\R)[2]\to E_{\R}[2] \xrightarrow{\sigma} E_{\R}[2] \to E_{\R}[2]_{\Gal(\C/\R)}  \to 0
\]
%
%
%There are natural maps $E(\R)\to V(E_{\R})$ and 
%$V(E_{\R}) \twoheadrightarrow E(\R)/E(\R)^0$, 
%where $E(\R)^0$ is the connected component of the identity. 
%The composition 
%\[
%E(\R)\to V(E_{\R}) \twoheadrightarrow E(\R)/E(\R)^0
%\]
%is the natural surjection (\cite[Prop.~2.42]{Blo81}). 
The assertion follows from 
\[
\dim_{\F_2}(V(E_v)/2)\le \dim_{\F_2}(E_{\R}[2]_{\Gal(\C/\R)}) = \dim_{\F_2}(E_{\R}(\R)[2])
\] 
and the structure theorem 
\[
E(\R)\simeq \begin{cases}
	\R/\Z, & \mbox{if $\Delta(E_{\R})<0$},\\
	\R/\Z\oplus \Z/2, & \mbox{if $\Delta(E_{\R})>0$}
\end{cases}
\]
(\cite[Chap.~V, Cor.~2.3.1]{Sil151}).
\end{proof}

\begin{lem}\label{lem:mult}
Let $K$ be a finite extension of the rational $l$-adic field $\Q_l$ 
with residue field $\F_K$, and 
$p$ a prime number. 
Let $E_K$ be an elliptic curve over $K$ which has split multiplicative reduction.  
\begin{enumerate}
	\item We have $\dim_{\F_p} (V(E_K)/p) \le 1$. 
	\item We suppose one of the conditions below: 
	\begin{enumerate}[label=$\mathrm{(\alph*)}$]
	\item $l = p$ and the ramification index satisfies $e_{K/\Qp} < p-1$.
	\item $l\neq p$ and $p\nmid (\#\FK-1)$. 
	\end{enumerate}
	Then, $V(E_K)/p = 0$.
	\item 
	Assume that the extension $K/\Q_l$ is abelian.  
	Put $M_K =\max\set{m |\mu_m \subset K}$ 
	and 
	\[
	M(E_K) =\frac{M_K}{\#(q_K,K^{\times})_{M_K}},
	\]
	where $q_K\in K^{\times}$ is the Tate parameter such that $E(\ol{K})\simeq \ol{K}^{\times}/q_K^{\Z}$ and 
	\[
	(-,-)_{M_K}\colon K^\times \times K^{\times} \to \mu_{M_K}
	\] 
	is the Hilbert symbol. 
	Then, 
	\[
	\dim(V(E_K)/p) = 
	\begin{cases}
	1, & \mbox{if $p\mid M(E_K)$},\\
	0, & \mbox{otherwise}.
	\end{cases}
	\]
\end{enumerate}
\end{lem}
\begin{proof}
(i) 
Put 
$M_K =\max\set{m |\mu_m \subset K}$.
It is known that the group 
$V(E_K)/V(E_K)_{\div}$ is finite and cyclic of order $M_K^{\ast}$ with $M_K^\ast \mid M_K$ 
(\cite[Prop.~2.1]{Hir22}, see also \cite[Thm.~1.2]{Asa06} for the case where $K/\Ql$ is abelian). 
We obtain $\dim_{\F_p} (V(E_K)/p) \le 1$. 

\sn 
(ii) 
\textbf{Case (a)}: $l =p$ and $e_{K/\Qp} < p-1$. 
If we assume $V(E_K)/p \neq 0$, then $p\mid M_{K}^{\ast} = \# (V(E_K)/V(E_K)_{\div})$. 
This implies $p\mid M_K$, and hence $\mu_p \subset K$. 
Since the extension $\Q_p(\mu_p)/\Qp$ is totally ramified extension of degree $p-1$ (\cite[Chap.~IV, Sect.~4, Prop.~17]{Ser68}), $e_{K/\Qp}\ge p-1$. 
This contradicts the assumption $e_{K/\Qp}<p-1$. 

\sn 
\textbf{Case (b)}: $l\neq p$, $p\nmid (\#\F_K-1)$. 
By \cite[Chap.~IV, Sect.~4, Cor.~1]{Ser68}, 
$[K(\mu_p):K]  = \min\set{r |(\#\F_{K})^r\equiv 1 \bmod p}$. 
This implies 
$[K(\mu_p):K]>1$ and $p\nmid M_{K}$. 
Since there is a surjective homomorphism 
$\Z/M_K\surj V(E_K)/V(E_K)_{\div}$
(\cite[Prop.~2.1]{Hir22}), $V(E_K)/p =0$. 

\sn
(iii) 
By  \cite[Thm.~1.2]{Asa06}, there is an isomorphism 
\[
V(E_K)/V(E_K)_{\div} \simeq \mu_{M_K}/(q_K,K^{\times})_{M_K}
\]
which is a cyclic group of order $M(E_K) = M_K/\#(q_K,K^{\times})_{M_K}$. 
The assertion follows from this. 
\end{proof}

\begin{lem}\label{lem:nonsp}
Let $K$ be a finite extension of the rational $l$-adic field $\Q_l$ 
with residue field $\F_K$, and 
$p$ an odd prime number. 
Let $E_K$ be an elliptic curve over $K$ which has non-split multiplicative reduction.  
\begin{enumerate}
	\item We have $\dim_{\F_p} (V(E_K)/p) \le 1$. 
	
	\item We suppose one of the conditions below: 
	\begin{enumerate}[label=$\mathrm{(\alph*)}$]
	\item $l = p$ and the ramification index satisfies $e_{K/\Qp} < p-1$.
	\item $l\neq p$ and $p \nmid ((\#\F_K)^2-1)$.
	\end{enumerate}
	Then, $V(E_K)/p = 0$.
\end{enumerate}
\end{lem}
\begin{proof}
There exists an unramified quadratic extension $K'/K$ 
such that the base change $E_{K'} = E_K\otimes_K K'$ has split multiplicative reduction 
(\cite[Thm.~5.3]{Sil151}, \cite[Appendix C, Thm.~14.1]{Sil106}).
%Moreover, there exists $q_{K'} \in {K'}^{\times}$ uniquely 
%such that $E_{K'}$ is isomorphic to the Tate curve 
%$\ol{K'}^{\times}/q_{K'}^{\Z}$. 
The composition of the restriction and the norm map gives 
\begin{equation}
\label{eq:resV}
	V(E_K)/p \xrightarrow{\res} V(E_{K'})/p \xrightarrow{N_{K'/K}} V(E_K)/p
\end{equation}
is the multiplication by $[K':K] = 2$. 
For $p>2$, the restriction $\res$ above is injective. 
The assertions (i) and (ii) follow from \autoref{lem:mult}. 
\end{proof}

%
%\begin{lem}\label{lem:add}
%Let $K$ be a finite extension of $\Q_l$ 
%with residue field $\F_K$ 
%and $p$ a prime number $>3$. 
%%We suppose that  the ramification index satisfies $e_{K/\Q_l}< l-1$. 
%Let $E_K$ be an elliptic curve over $K$ which has additive reduction.  
%Then, 
%\[
%\dimFp(V(E_K)/p) \le \dimFp (V(E_{L})/p)
%\]
%\end{lem}

\subsection*{Mod \texorpdfstring{$p$}{p} Galois representations}
Finally, 
we study the third component $E[p]_{G_F}$ in the exact sequence of \autoref{thm:fin}.
Recall that 
the $G_F$-coinvariant quotient is given by $E[p]_{G_F} = E[p]/I(E[p])$, 
where 
$I(E[p])$ is the subgroup of $E[p]$ generated by elements of the form $\sigma P - P$ 
for $\sigma \in G_F$ and $P\in E[p]$. 
By using a classification of the image $\Im(\rho_{E,p})$ 
of the mod $p$ Galois representation 
\[
\rho_{E,p}\colon G_F\to \Aut(E[p])
\]
associated to $E[p]$, 
we investigate $E[p]_{G_F}$.

\begin{lem}\label{lem:surj}
	Assume that there exists a basis of $E[p]$ such that 
	the image of $\rho_{E,p}\colon G_F\to \Aut(E[p])\simeq GL_2(\Fp)$ contains $SL_2(\Fp)$.
	Then, $E[p]_{G_{F}} =0$. 
\end{lem}
\begin{proof}
	Take a basis $\set{P,Q}$ of $E[p]$ and identify $\Aut(E[p]) \simeq GL_2(\Fp)$. 
	%We assume that the image of $\rho_{E,p}$ contains $SL_2(\Fp)$. 
	Corresponding to $\pmat{1 & 0 \\ 1 & 1}, \pmat{1 & 1 \\ 0 & 1} \in SL_2(\Fp)$, 
	there exist $\sigma, \tau \in G_F$ such that 
	$\sigma P = P+Q, \sigma Q = Q$ and $\tau P = P, \tau Q = P+Q$. 
	Then $\sigma P-P = Q, \tau Q -Q = P$ imply $P,Q \in I(E[p])$. Hence, $E[p] = I(E[p])$.
\end{proof}
%From the above lemma, when $\rho_{E,p}$ is surjective,  
%$V(E)$ is $p$-divisible by \autoref{thm:pdiv}.

\begin{lem}\label{lem:even}
For the even prime $p=2$, we have 
\[
\dim_{\F_2}(E[2]_{G_F}) = \begin{cases}
	0 & \mbox{if $E(F)[2]=0$},\\
	1 & \mbox{if $E(F)[2]\neq 0$ and $\Delta(E)\not\in F^2$},\\
	2 & \mbox{if $E(F)[2] \neq 0$ and $\Delta(E)\in F^2$}.
\end{cases}
\]
\end{lem}
\begin{proof}
First, we consider the case $E(F)[2]=0$.
By \autoref{lem:surj}, we may assume that $\rho_{E,2}$ is not surjective. 
By \cite[Prop.~2.1]{RV01}, 
for some basis $\set{P,Q} $ of $E[2]$, 
the image of $\rho_{E,2}$ is generated by 
$\pmat{0 & 1 \\ 1 & 1}$, the cyclic subgroup of order $3$. 
Corresponding to $\pmat{0 & 1 \\ 1 & 1}$ and $\pmat{0 & 1 \\ 1 & 1}^2 = \pmat{1 & 1 \\ 1 & 0}$, 
there exist $\sigma,\tau \in G_F$ such that 
$\sigma P = Q, \sigma Q = P+Q$ and 
$\tau P = P+Q, \tau Q = P$.   
%$\pmat{1 & 1 \\ 1 & 0}}$. 
Therefore, $P = \sigma Q  - Q$ and $Q = \tau P - P$ are in $I(E[2])$. 
We obtain $E[2] = I(E[2])$. 

Next, we consider the case where $E(F)[2]\neq 0$ and $\Delta(E)\not\in F^2$. 
In this case, 
there is a basis $\set{P,Q} $ of $E[2]$, 
such that 
the image of $\rho_{E,2}$ coincides with 
the cyclic subgroup of order 2 generated by 
$\pmat{1 & 1 \\ 0 & 1}$ (\cite[Prop.~2.1]{RV01}). 
There is $\sigma\in G_F$ such that $\sigma P = P$ and $\sigma Q = P + Q$. 
We have $P\not \in I(E[2])$ while $Q\in I(E[2])$, hence $\dim_{\F_2}(E[2]_{G_F}) = 1$. 

Finally, suppose $E(F)[2]\neq 0$ and $\Delta(E)\in F^2$. 
By \cite[Prop.~2.1]{RV01} again, 
the image of $\rho_{E,2}$ is trivial so that $I(E[2]) = 0$.
\end{proof}

For an odd prime $p$, we consider the following conditions:
\begin{enumerate}
	\item[($\mathrm{SC}_p)$] $\dim_{\Fp}(E(F)[p]) = 1$, and $E$ has more than one $F$-isogeny of degree $p$.
	\item[($\mathrm{B}'_p)$] $\dim_{\Fp}(E(F)[p]) = 1$, and $E$ has only one $F$-isogeny of degree $p$.
	\item[$(\mathrm{B}_p)$] $E(F)[p]=0$ and there exists an $F$-isogeny $\phi\colon E'\to E$ of degree $p$ with $E'(F)[p] \neq 0$.
\end{enumerate}
The first condition $(\mathrm{SC}_p)$ indicates that 
the image of $\rho_{E,p}$ is split Cartan, 
and in the other cases $(\mathrm{B}_p)$ and $(\mathrm{B}'_p)$, 
the image of $\rho_{E,p}$ is contained in a Borel subgroup in the sense of \cite[Sect.~2]{Ser72}. 

%\begin{rem}
%For an odd prime $p$ and a semi-stable elliptic curve $E$ over $\Q$, 
%if the mod $p$ Galois representation $\rho_{E,p}\colon G_\Q \to \Aut(E[p])$ is not surjective, then 
%one of the three conditions $(\mathrm{SC}_p)$, $(\mathrm{B}'_p)$, or $(\mathrm{B}_p)$ holds 
% (\cite[Prop.~1]{Ser96}). 
%\end{rem}

\begin{lem}\label{lem:quot}
Let $p$ be an odd prime.

\begin{enumerate}
	\item 
Assume $\mu_p\not\subset F$. 
Then 
\[
\dim_{\Fp}(E[p]_{G_F})= \begin{cases} 
0,& \mbox{if $(\mathrm{B}'_p)$},\\ 
1,& \mbox{if $(\mathrm{SC}_p)$ or $(\mathrm{B}_p)$ holds},\\
2,& \mbox{if $E[p]\subset E(F)$}.
\end{cases}
\]
\item 
Assume $\mu_p\subset F$. 
Then 
\[
\dim_{\Fp}(E[p]_{G_F})= \begin{cases} 
1,& \mbox{if $(\mathrm{B}'_p)$, $(\mathrm{SC}_p)$ or $(\mathrm{B}_p)$ holds},\\
2,& \mbox{if $E[p]\subset E(F)$}.
\end{cases}
\]

\end{enumerate}

\end{lem}
\begin{proof}
	First, we consider the case $E[p]\subset E(F)$. 
	Since $\rho_{E,p}$ is trivial, $I(E[p]) = 0$ and hence $\dimFp(E[p]_{G_F}) = \dim(E[p]) = 2$. 

	Next, we suppose $\dimFp(E(F)[p])\le 1$. 
	By the Weil pairing, $\det(\rho_{E,p}(\sigma)) = \chi_p(\sigma)$ 
	for all $\sigma\in G_F$, where $\chi_p:G_F\to \Aut(\mu_p)=\Fp^{\times}$ is the mod $p$ cyclotomic character. 
	By \cite[Prop.~1.2, Prop.~1.4]{RV01}, there exists a basis $\set{P,Q}$ of $E[p]$ 
	such that  
	\begin{equation}\label{eq:Imrho}		
	\Im(\rho_{E,p}) = 
		\begin{cases}
		\pmat{1 & 0 \\ 0 & \Im(\chi_p)} & \mbox{if $(\mathrm{SC}_p)$ holds},\\
		\pmat{1 & \ast \\ 0 & \Im(\chi_p)} & \mbox{if $(\mathrm{B}'_p)$ holds}, \\
		\pmat{\Im(\chi_p) & \ast \\ 0 & 1} & \mbox{if $(\mathrm{B}_p)$ holds},
		\end{cases}
	\end{equation}
	through the isomorphism $\Aut(E[p]) \simeq GL_2(\Fp)$. 
		
%	\sn 
%	\textbf{Case $(\mathrm{SC}_p)$}:
%	We consider the case $(\mathrm{SC}_p)$. 
%	As the point $P$ satisfies $\sigma P = P$ for all $\sigma\in G_F$, 
%	$I(E[p])$ is generated by the elements of the form $\sigma Q - Q = \chi_p(\sigma)Q$ for $\sigma \in G_{F}$ so that $\dim_{\Fp}(I(E[p]) = 1$. 
%	As a result, $\dimFp(E[p]_{G_F}) = 1$. 
%		
%	\sn 
%	\textbf{Case $(\mathrm{B}_p)$}: 
%	We assume the condition $(\mathrm{B}_p)$. 
%	For any $\sigma\in G_F$, $\sigma P - P= (\chi_p(\sigma)-1)P \in I(E[p])$ 
%	and $\sigma Q - Q  = a P$ for some $a\in \Fp$.
%	$I(E[p])$ is generated by $P$. We obtain $\dimFp(I(E[p])) = 1$ and $\dimFp(E[p]_{G_F}) =1$. 
%	
%	\sn 
%	\textbf{Case $(\mathrm{B}'_p)$}:
%	We suppose $(\mathrm{B}'_p)$. 
%	If $\mu_p\subset F$, then $\chi_p$ is trivial. 
%	For any $\sigma \in G_F$, $\sigma P = P$ and $\sigma Q  = a P + Q$ for some $a\in \Fp$. 
%	This implies that $I(E[p])$ is generated by $P$ and hence $\dimFp(E[p]_{G_F}) = 1$. 
%	
%	Next, consider the case $\mu_p\not\subset F$. 
%	By taking a generator $r\neq 1$ of $\Im(\chi_p)\subset \Fp^{\times}$, 
%	the image of $\rho_{E,p}$ contains $\pmat{1 & 0 \\ 0 & r}$ and $\pmat{1 & 1 \\ 0 & 1} \in GL_2(\Fp)$. 
%	There exist $\sigma, \tau \in G_{\Q}$ such that $\sigma Q = rQ$ and $\tau Q = P+Q$. 
%	Therefore, $\sigma Q - Q = (r-1)Q$ and $\tau Q- Q = P$. 
%	Both $P$ and $Q$ are contained in $I(E[p])$ and hence $E[p]_{G_{F}} = E[p]/I(E[p]) = 0$. 

	By considering the dual representation $\rho_{E,p}^\vee$ 
	and $(E[p]_{G_F})^{\vee}\simeq (E[p]^\vee)^{G_F}$ (\cite[Chap.~II, Thm.~2.6.9]{NSW08}), 
	we determine the dimension of the $G_F$-invariant space $(E[p]^\vee)^{G_F}$. 
	Note that 
	the action of $\sigma \in G_F$ on $E[p]^{\vee}$ is given by the contragredient matrix $(\rho_{E,p}(\sigma^{-1}))^T$ with respect to 
	the dual basis $\set{\phi_P,\phi_Q}$ for $E[p]^{\vee}$ of the basis $\set{P,Q}$. 

	\sn 
	\textbf{Case $(\mathrm{SC}_p)$}:
	We consider the case $(\mathrm{SC}_p)$. 
	As $\rho_{E,p}$ is non-trivial, so is $\chi_p$. 
	By \eqref{eq:Imrho}, for any $\sigma \in G_F$, 
	we have $\sigma\phi_P = \phi_P$ and $\sigma\phi_Q = \chi_p^{-1}(\sigma)\phi_Q$. 
	This implies $(E[p]^{\vee})^{G_F}$ is generated by $\phi_P$ and hence $\dimFp((E[p]^\vee)^{G_F}) = \dimFp(E[p]_{G_F}) = 1$. 
	
	\sn 
	\textbf{Case $(\mathrm{B}_p)$}: 
	We assume the condition $(\mathrm{B}_p)$. 
	For any $\sigma\in G_F$, 
	we have $\sigma\phi_P = \chi_p^{-1}(\sigma)\phi_P +a \phi_Q$ for some $a \in \Fp$ 
	and $\sigma \phi_Q = \phi_Q$ so that 
	$\dimFp((E[p]^\vee)^{G_F}) = \dimFp(E[p]_{G_F}) = 1$. 

	\sn 
	\textbf{Case $(\mathrm{B}'_p)$}:
	We suppose $(\mathrm{B}'_p)$. 
	If $\mu_p\subset F$, then $\chi_p$ is trivial. 
	For any $\sigma \in G_F$, $\sigma \phi_P = \phi_P + a \phi_Q$ for some $a\in \Fp$ 
	and $\sigma \phi_Q  = \phi_Q$. We obtain $\dimFp(E[p]_{G_F}) = 1$.
	Consider the case $\mu_p\not\subset F$. 
	For any $\sigma \in G_F$, $\sigma \phi_P = \phi_P + a \phi_Q$ for some $a\in \Fp$ 
	and $\sigma \phi_Q  = \chi_p^{-1}(\sigma)\phi_Q$. 
	This implies $(E[p]^{\vee})^{G_F} = 0$ and hence $\dimFp(E[p]_{G_F}) = 0$.
\end{proof}

\section{Elliptic curve over \texorpdfstring{$\Q$}{Q}}
\label{sec:ECQ}
In this section, the kernel $\Ker(\dbar_{E,p})$ and the cokernel $\Coker(\dbar_{E,p})$ are examined in more detail by applying the main results of the previous section to the case $F=\Q$.
Until the end of this note, let $E$ be an elliptic curve defined over $\Q$. 

%As referred to in Introduction (\autoref{rem:Mazur}), 
%the mod $p$ Galois representation $\rho_{E,p}\colon G_\Q\to \Aut(E[p])$ is known to be surjective 
%for all prime $p\ge 11$.

\begin{lem}\label{lem:SCp}

\begin{enumerate}
	\item 
	If we assume that $(\mathrm{SC}_p)$ holds for $E$ and some odd prime $p$, then 
	the map 
	\[
	\dbar_{E,p}\colon V(E)/p\to  \bigoplus_{l\in \Sigma_\good(E)}\ol{E_l}(\F_l) /p
	\] 
	is surjective. 
	
	\item If we assume $E(\Q)[2]\neq 0$ and $\Delta(E)\in \Q^2$, 
	then $\dbar_{E,2}$ is surjective. 
	\end{enumerate}
\end{lem}
\begin{proof}
	(i)
	For each $l\in \Sigma_{\good}(E)$,  
	consider the composition 
	\[
	\dbar_{E,p}^{(l)}\colon V(E)/p\xrightarrow{\dbar_{E,p}}  \bigoplus_{l\in \Sigma_\good(E)}\ol{E_l}(\F_l) /p\xrightarrow{\mathrm{projection}} \ol{E_l}(\F_l)/p.
	\]
	By the construction (\Cf\eqref{eq:dX}), 
	and the isomorphism $V(E) \simeq K(\Q;E,\Gm)$ (\Cf \eqref{eq:Som}), 
	the map $\dbar_{E,p}$ is given by 
	\[
	\dbar_{E,p}^{(l)} (\set{P, f}_{F/\Q}) = \sum_{v\mid l}  v(f)N_{\F_v/\F_l}(\Pbar_v) 
	\]
	for $f\in F^{\times}$ and $P\in E(F)$, 
	where the place 
	$v$ is considered as the valuation map  $v\colon F^{\times} \to \Z$ corresponding to $v\mid l$, 
	$\F_{v}$ is the residue field of the local field $F_v$, 
	and 
	$\Pbar_v \in \ol{E_v}(\F_v)$ is the image of the reduction map $E(F)\hookrightarrow E_v(F_v) \to \ol{E_v}(\F_v)$ of $P$ at $v$. 

	Take a non-zero $P \in E(\Q)[p]$. 
	Put $F = \Q(E[p])$ and consider a basis $\set{P,Q}$ of $E(F)[p]$ with $Q\not\in E(\Q)[p]$. 
%	As $Q \not\in E(\Q)[p]$, 
%	the prime $p$ should be $3,5$ or $7$ by Mazur's theorem (\Cf\eqref{eq:Maz}). 
	The image of $\rho_{E,p}$ is 
	\[
	\pmat{1 & 0 \\ 0  & \Im(\chi_p)}
	\]
	(\Cf\eqref{eq:Imrho}). 
	The mod $p$ character $\chi_p$ is surjective, $F = \Q(\mu_p)$ and $[F:\Q] = p-1$.
	Consider the short exact sequence of finite groups 
	\[
	0 \to \ol{E_l}(\Fl)[p]\to \ol{E_l}(\Fl)\xrightarrow{p}\ol{E_l}(\Fl)\to \ol{E_l}(\Fl)/p\to 0.
	\]
	By counting the orders, we have 
	\begin{equation}
		\label{eq:dim}
		\dimFp( \ol{E_l}(\Fl)[p]) = \dim_{\Fp}(\ol{E_l}(\Fl)/p).
	\end{equation}
	
%	\sn
%	(The case:$\dim_{\Fp}E(\Q)[p]=2$).
%	Suppose $l \neq p$ and $\dim_{\Fp}E(\Q)[p]=2$ and hence $Q\in E(\Q)[p]$. 
%	The prime $p$ should be $2$ by Mazur's theorem (\Cf\eqref{eq:Maz}).   
%	Since the reduction map $E_l(\Q_l)[p]\to \ol{E_l}(\Q_l)$ is injective (\cite[Chap.~VII, Prop.~3.1]{Sil106}), 
%	$\dim_{\Fp}\ol{E_l}(\F_l)[p] = \dim_{\Fp}\ol{E_l}(\F_l)/p = 2$. 
%	$\ol{E_l}(\F_l)/p$ is generated by the images $\Pbar_l$ and $\Qbar_l$ in $\ol{E_l}(\F_l)$ 
%	by the reduction map. 
%	By $\dbar_{E,p}^{(l)}(\set{P,l}_{\Q/\Q}) = \Pbar_l$ and 
%	$\dbar_{E,p}^{(l)}(\set{Q,l}_{\Q/\Q}) = \Qbar_l$, $\dbar_{E,p}^{(l)}$ is surjective. 
%	
%	For $l=p$, 

	\sn
	\textbf{Case $l\neq p$}:  
	The reduction map $\red_l\colon E_l(\Ql)\to \ol{E_l}(\Fl)$ gives 
	a commutative diagram with exact rows: 
	\[
	\xymatrix{
	0 \ar[r] & \Ehat_l(l\Z_l)\ar[r]\ar[d]^p &  E_l(\Ql)\ar[r]^{\red_l}\ar[d]^p & \ol{E_l}(\Fl)\ar[r]\ar[d]^p & 0\\
	0 \ar[r] & \Ehat_l(l\Z_l)\ar[r] &  E_l(\Ql)\ar[r]^{\red_l} & \ol{E_l}(\Fl)\ar[r] & 0\\
	}
	\]
	where $\Ehat_l(l\Z_l)$ is the group associated to the formal group law $\Ehat_l$ of $E_l$ 
	(\cite[Chap.~VII, Prop.~2.1, Prop.~2.2]{Sil106}). 
	By the snake lemma, there is a long exact sequence 
	\begin{align*}
	0& \to\Ehat_l(l\Z_l)[p] \to  E_l(\Ql)[p]\xrightarrow{\red_l} \ol{E_l}(\Fl)[p] \\
	 &\xrightarrow{\delta} \Ehat_l(l\Z_l)/p \to E_l(\Ql)/p\xrightarrow{\red_l} \ol{E_l}(\Fl)/p\to 0.
	\end{align*}
	Since $\Ehat_l(l\Z_l) \simeq l\Z_l$ (\cite[Chap.~IV, Thm.~6.4]{Sil106}), 
	and $\Ehat_l(l\Z_l)[p] = \Ehat_l(l\Z_l)/p = 0$. 
	We obtain 
	\[
	\dim_{\Fp} (E_l(\Ql)[p] )= \dimFp( \ol{E_l}(\Fl)[p]) \stackrel{\eqref{eq:dim}}{=} \dim_{\Fp}( \ol{E_l}(\Fl)/p) = \dim_{\Fp}(E_l(\Ql)/p).
	\]
	Take a place $v \mid l$ of $F$. 
	For the reduction map $E_v(F_v)[p]\to \ol{E_v}(\F_v)$ is injective (\cite[Chap.~VII, Prop.~3.1]{Sil106}), 
	$\dimFp(\Evbar(\Fv)[p]) = \dimFp(\Evbar(\Fv)/p) = 2$. 
	
	Consider the case where the extension $F/\Q$ is completely split at $l$. 
	We have 
	$E_v(F_v)[p] = E_l(\Ql)[p] \simeq \ol{E_l}(\Fl)[p]$. 
	The group $\ol{E_l}(\Fl)/p$ is generated by $\Pbar_l$ and $\Qbar_l$ 
	the images of $P$ and $Q$ by the reduction map $\red_l$. 
	The equality 
	\[
		\dbar_{E,p}^{(l)}(\set{P,l}_{\Q/\Q}) = \Pbar_l
	\] 
	holds and the projection formula gives 
	\[
	\dbar_{E,p}^{(l)}(\set{Q,l}_{F/\Q})  = \sum_{v\mid l}\Qbar_l = (p-1)\Qbar_l.
	\] 
	The map $\dbar_{E,p}^{(l)}$ is surjective. 
	
	Next, we assume that the extension $F/\Q$ is not completely split at $l$. 
	The extension $F/\Q$ is unramified at $l \neq p$. 
	In particular, $l\not\equiv 1 \bmod p$. 
	Since the reduction map $\red_l \colon E_l(\Q_l)[p]\hookrightarrow \ol{E_l}(\F_l)[p]$ is injective, 
	the image $\Pbar_l = \red_l(P)$ of $P\in E(\Q)[p]$ is non-zero. 
	We have 
	\[
		\dbar_{E,p}^{(l)}(\set{P,l}_{\Q/\Q}) = \Pbar_l, 
	\] 
	and $\dimFp(\ol{E_l}(\Fl)/p)\ge 1$. 
	To show $\dim_{\Fp}(\ol{E_l}(\Fl)[p]) = 1$,  
	we assume $\dim_{\Fp}(\ol{E_l}(\Fl)[p]) = 2$. 
	Then, $\dim_{\Fp} (E_l(\Ql)[p]) = 2$. 
	Take the place $v$ of $F$ above $l$, there is a commutative diagram:
	\[
	\xymatrix{
	E(F)[p]\ar[d]^{N_{F/\Q}} \ar[r]^{\simeq} & E_v(F_v)[p]\ar[r]^{\simeq}\ar[d]^{N_{F_v/\Q_l}} & \ol{E_v}(\F_v)[p]\ar[d]^{N_{\F_v/\Fl}} \\ 
	E(\Q)[p] \ar@{^{(}->}[r] & E_l(\Ql)[p]\ar[r]^{\simeq} & \ol{E_l}(\F_l)[p] 
	}
	\]
	In the above diagram, the vertical maps are surjective because 
	$[F:\Q] = p-1$. 
	Therefore, the norm maps $N_{F_v/\Ql}$ and $N_{\F_v/\Fl}$ are bijective. 
	In particular, $N_{F_v/\Ql}(Q) \neq 0$ in $E_l(\Ql)[p]$.  
	This implies $N_{F/\Q}(Q) \neq 0$ in $E(\Q)[p]$.  
	The points $P, N_{F/\Q}(Q)$ are linearly independent. 
	This contradicts $\dimFp(E(\Q)[p]) = 1$. 
%	As we have $E_l(\Ql)[p]  \subset E_v(F_v)[p]$. 
%	The $F$-rational point $Q$ is in $E(F)[p] \cap E(\Ql)[p]$. 
%	We have $\zeta_p \not \in \Q_l$ 
%	because of $p\nmid (l-1)$ (\cite[Chap.~IV, Sect.~4, Cor.~1]{Ser68}). 
%	We obtain $F\cap \Q_l = \Q$ and $Q\in E(F)[p] \cap E(\Ql)[p] = E(\Q)[p]$. 
%	This contradicts $\dimFp( E(\Q)[p]) =1$. 
%	By $\dbar_{E,p}^{(l)}(\set{P,l}_{\Q/\Q}) = \Pbar_l$, 
%	the map $\dbar_{E,p}^{(l)}$ is surjective. 
	
	\sn 
	\textbf{Case $l = p$}: 
	The extension $F/\Q$ is totally ramified at $p$. 
	When $E$ has good supersingular reduction at $p$, 
	$\ol{E_p}[p] = 0$ and hence $\ol{E_p}(\Fp)/p =  0$. 
	We may assume that $\ol{E_p}$ is ordinary. 
	Consider the following exact sequence as above: 
	\begin{align*}
	0&\to \Ehat_p(p\Zp)[p] \to  E_p(\Qp)[p]\xrightarrow{\red_p} \ol{E_p}(\Fp)[p] \\
	&\xrightarrow{\delta} \Ehat_p(p\Zp)/p \to E_p(\Qp)/p\xrightarrow{\red_p} \ol{E_p}(\Fp)/p\to 0.
	\end{align*}
	By the formal logarithm $\Ehat_p(p\Zp) \simeq p\Zp$ (\cite[Chap.~IV, Thm.~6.4]{Sil106}), 
	we have $\Ehat_p(p\Zp)[p] = 0$ and $E_p(\Qp) \simeq \Zp\oplus E_p(\Qp)_{\tor}$ (\Cf\cite[Lem.~1]{Hir19}). 
	By the Hasse bound (\cite[Chap.~V, Thm.~1.1]{Sil106}), there are inequalities 
	$\#\ol{E_p}(\Fp) < 2\sqrt{p} + p+1< p^2$ and hence 
	\begin{align*}
		&\dimFp(\ol{E_p}(\Fp)[p]) = \dimFp (\ol{E_p}(\Fp)/p) = 1,\\
		&\dimFp(\Ehat_p(p\Zp)/p) = 1,\quad \mbox{and}\\
		&\dimFp(E_p(\Qp)/p) = \dimFp(E_p(\Qp)[p]) +1 = 2. 
		\end{align*}
	The rational point $P \in E(\Q)[p]$ generates $E_p(\Qp)[p]$. 
	Since the reduction map $\red_p\colon E_p(\Qp)[p]\to \ol{E_p}(\Fp)[p]$ 
	is injective, $\ol{E_p}(\Fp)[p]$ is generated by $\Pbar_p = \red_p(P)$. 
	In particular, $\Pbar_p\neq 0$ in $\ol{E_p}(\Fp)$. 
	The equality $\dbar_{E,p}^{(p)} (\set{P,p}_{\Q/\Q}) = \Pbar_p$ indicates that 
	$\dbar_{E,p}^{(p)}$ is surjective. 
	
	To show the assertion, 
	take any element $\ol{R} = \sum_l \ol{R}_l$ in $\bigoplus_{l\in \Sigma_{\good}(E)}\ol{E_l}(\Fl)/p$ with $\ol{R}_l\in \ol{E_l}(\Fl)/p$. 
	There is a finite set of primes $S\subset \Sigma_{\good}(E)$ such that 
	$\ol{R}_l= 0$ for any $l\in \Sigma_{\good}(E)\ssm S$. 
	Hence,  
	\[
	\sum_{l\in S}\dbar_{E,p}(\set{P,l}_{\Q/\Q}),\quad \mbox{and}\quad \sum_{l\in S}\dbar_{E,p}(\set{Q,l}_{F/\Q})
	\]
	generates $\ol{R}$. 
	
	\sn
	(ii) 
	In the case $E(\Q)[2]\neq 0$ and $\Delta(E)\in \Q^2$, 
	the mod $2$ Galois representation $\rho_{E,2}$ is trivial 
	so that $E(\Q)[2] = E[2]$. 
	Take a basis $\set{P,Q}$ of $E(\Q)[2]$. 
	The equalities  
	\[
	\dbar_{E,2}^{(l)}(\set{P,l}_{\Q/\Q}) = \Pbar_l,\quad\mbox{and},\quad \dbar_{E,2}^{(l)}(\set{Q,l}_{\Q/\Q}) = \Qbar_l
	\]
	implies the assertion.
\end{proof}

Recalling from \autoref{lem:mult}, for a finite extension $K/\Ql$, 
we put
\[
M_K = \max\set{m |\mu_m\subset K}, 
\]
and $\dt_K\colon K^{\times}\times K^{\times}\to \mu_{M_K}$ is the tame symbol map defined by 
\begin{equation}
\label{eq:tame}	
\dt_K(a,b) = (-1)^{v_K(a)v_K(b)}\dfrac{b^{v_K(a)}}{a^{v_K(b)}}\ \bmod \m_K.
\end{equation}

\begin{lem}\label{lem:multQ}
Let $E_{l}$ be an elliptic curve over $\Q_l$ which has multiplicative reduction,
and $p$ a rational prime. 
\begin{enumerate}
	\item If $l=p$ and $p>2$, then $V(E_p)/p = 0$. 
	\item If $l\neq p$ and $E_l$ has split multiplicative reduction, then 
	we have 
	\[
	\dimFp(V(E_l)/p) = 
	\begin{cases}
	1, & \mbox{if $l-1 \equiv \dfrac{l-1}{\#\dt_{\Ql}(q_l,\Ql^{\times})} \equiv 0 \bmod p$}\\
	0, & \mbox{otherwise},
	\end{cases}
	\]
	where $q_l \in \Ql^\times$ is the Tate parameter of $E$.
	\item 
	If $l\neq p$, $E_l$ has non-split multiplicative reduction, 
	and assume $p\nmid l^2-1$, or $p\nmid \dfrac{l^2-1}{\#\dt_{K}(q_K,K^{\times})}$, 
	where $K/\Ql$ is a quadratic extension such that $E_K$ has split multiplicative reduction,  
	and $q_K \in K^\times$ is the Tate parameter of $E_K$. 
	Then $\dimFp(V(E_l)/p) =0$  
%	\le 
%	\begin{cases}
%	1, & \mbox{if $l^2-1\equiv \dfrac{l^2-1}{\#\dt_{K}(q_K,K^{\times})} \equiv 0 \bmod p$}\\
%	0, & \mbox{otherwise},
%	\end{cases}
	\end{enumerate}
\end{lem}
\begin{proof}
	(i) This follows directly from \autoref{lem:mult} (ii) (the case (a)).
	
	\noindent 
	(ii) By \autoref{lem:mult} (ii), 
	if $p\mid l-1$ then $V(E_l)/p = 0$. 
	In particular, we may assume $l>2$. 
	%in the case $l=2$, we have $V(E_2)/p = 0$. 
	By \cite[Chap.~IV, Sect.~4, Prop.~17]{Ser68}, 
	we have $M_{\Q_l} = l-1$. 
	From \autoref{lem:mult} (iii), 
	the order $M({E_l})$ of the finite cyclic group $V(E_l)/V(E_l)_{\div}$ is written 
	by the Hilbert symbol 
	\[
	M_{\Ql}^{\ast} = \dfrac{l-1}{\#(q_l,\Q_l^{\times})_{l-1}}.
	\]
	The Hilbert symbol coincides with the tame symbol map (\cite[Chap.~IV, (5.3)]{FV02}).
	
	\noindent
	(iii) Take the unramified quadratic extension $K/\Ql$ such that $E_K$ has split multiplicative reduction. 
	In the same way as above, $M_K = l^2-1$ 
%	In fact, when $l=2$, 
%	$[K(\mu_3):K] = \max\set{r | \# (\F_K)^r \equiv 1\bmod 3} = 1$ 
%	and $\mu_4 \not\in K$ because of $K/\Q_2$ is unramified. 
%	Hence, $M_K = 3$.  
%	For $l>2$, 
%	$[K(\mu_{l^2-1}):K] = 1$ and $1 = [K(\mu_{M_K}):K] = \max\set{r | (\F_K)^r\equiv 1 \bmod M_K}$ 
%	implies $l^2 -1 = M_K$. 
	and 
	\[
	M_K^{\ast} = \dfrac{l^2-1}{\#\dt_K(q_K,K^{\times})}.
	\]
	The assertion follows from the injection 
	$\res_{K/\Ql}\colon V(E_l)/p\hookrightarrow V(E_K)/p$. 
\end{proof}

\begin{rem}\label{rem:tame}
	For an elliptic curve $E_l$ over $\Ql$ which has split multiplicative reduction, 
	the image of the tame symbol map 
	$\dt_{\Ql}(q_l,\Ql^{\times})$ 
	is determined as follows: 
	If $l=2$, then $\dimFp(V(E_2)/p) = 0$ by \autoref{lem:multQ} so that 
	we consider the case $l>2$. 
	Let $\Delta(E_l)$ be the discriminant of $E_l$. 
	Its $l$-adic valuation coincides with that of the Tate parameter $q_l$ of $E_l$: 
	$m := v_l(\Delta(E_l)) = v_l(q_l)$. 
	By the structure theorem  
	$\Z_l^{\times} \simeq \mu_{l-1} \times  (1+l\Z_l)$ (\cite[Chap.~II, Prop.~5.3]{Neu99}) 
	and the unit group $1 + l\Z_l$ is $(l-1)$-divisible (\cite[Chap.~I, (5.5) Cor.]{FV02}). 
	There exists $n$ such that $q_l/l^{m} = \zeta^n v^{l-1}$, 
	for some $v\in 1+l\Z_l$ and a primitive $(l-1)$-root of unity $\zeta$. 
	Note that $z = \zeta \bmod l \in (\Z/l)^{\times}$ is a primitive root of modulo $l$.
	It is easy to see  
	\[
	\dt_{\Ql}(q_l,\Q_l^{\times}) = \dt_{\Ql}(l,\Ql^{\times})^m\dt_{\Ql}(r,\Ql^\times)^n =
	\mu_{l-1}^m\mu_{l-1}^n \subset \mu_{l-1}.
	\]
	Therefore, 
	$V(E_l)/p \simeq \Z/\gcd(p,m,n)$ and hence 
	\[
	\dimFp(V(E_l)/p) = 	
	\begin{cases}
	1, & \mbox{if $l-1 \equiv m \equiv n \equiv 0 \bmod p$},\\
	0, & \mbox{otherwise}.	
	\end{cases}
	\]
	
	For example, 
	let $E^{(2)}$ be the elliptic curve over $\Q$ 
	with Cremona label 651e2 referred in \autoref{ex:651}. 
	We have $\Delta(E^{(2)}) = -1\cdot 3^3\cdot 7^3 \cdot 31^3$. 
	The mod $p$ Galois representation $\rho_{E,p}$ is surjective for all $p\neq 3$. 
	For the remained prime $p=3$, 
	we determine the dimension $\dim_{\F_3}(V( E^{(2)}_{l})/3)$ 
	of the base change $E^{(2)}_l := E^{(2)} \otimes_\Q \Ql$ 
	for the bad primes $l=3,5$ and $7$. 
	For $l=3$, $\dim_{\F_3}(V( E^{(2)}_{3})/3) = 0$ because of $l=p$ (\autoref{lem:multQ}). 
	For $l = 7$ and $31$, the Tate parameters are of the form 
	\[
	 q_7 = 6\cdot 7^3 + \cdots,\quad q_{31} = 8\cdot 31^3 + \cdots .
	\]
	As $6 = 3^3$ in $(\Z/7)^{\times}$ and $8 = 3^{12}$ in $(\Z/31)^{\times}$, 
	we obtain 
	\[
	\dim_{\F_3}(V(E^{(2)}_{7})/3) = \dim_{\F_3}(V( E^{(2)}_{31})/3) = 1.
	\] 
	
	On the other hand, 
	let $E^{(3)}$ be the elliptic curve over $\Q$ 
	with Cremona label 651e2. 
	By $\Delta(E^{(3)}) = -1\cdot 3\cdot 7 \cdot 31$, 
	$\dim_{\F_3}(V( E^{(3)}_{l})/3) = 0$ for all $l = 3,7$ and $31$. 
\end{rem}

\begin{thm}
\label{thm:EQ}
	Let $E$ be an elliptic curve over $\Q$. 
	If $E[p]_{G_{\Q}} \neq 0$ for some odd prime $p$, then 
	there is an exact sequence 
	\[
	0 \to \Ker(\dbar_{E,p})\to  \bigoplus_{l \in \Sigma_\bad(E)}V(E_l)/p  \\
		\to \Z/p\Z \to \Coker(\dbar_{E,p})\to 0. 
	\]
	If we further assume $E$ satisfies $(\mathrm{SC}_p)$, then 
	$\Coker(\dbar_{E,p}) =0$.
\end{thm}
\begin{proof}%[{Proof of \autoref{thm:introEQ}}]
%From the assumption $E(\Q)[2] = 0$, $V(E)$ is $2$-divisible by \autoref{lem:even}. 
%First, we consider the case 
%where $\rho_{E,p}$ is surjective for all odd prime $p$. 
%By \autoref{lem:surj} and \autoref{thm:pdiv}, $V(E)$ is divisible. 
%
%Next,  
%we assume $\rho_{E,p}$ is not surjective for some odd prime $p$. 
%By Mazur's theorem (\Cf\autoref{rem:Mazur}), 
%there exists $p \in \set{3,5,7}$ such that $\rho_{E,l}$ is surjective for all $l\neq p$. 
%By \autoref{lem:surj} and \autoref{thm:pdiv}, $V(E)$ is $l$-divisible for all $l\neq p$. 
%In the case $(\mathrm{B}'_p)$, 
%$E[p]_{G_{\Q}} = 0$ (\autoref{lem:quot}) and hence $V(E)$ is $p$-divisible (\autoref{thm:pdiv}).
%
%
%In the cases $(\mathrm{SC}_p)$ and $(\mathrm{B}_p)$, 
%$E[p]_{G_{\Q}} \simeq \Z/p$ (\autoref{lem:quot}). 
By \autoref{thm:fin} (and \autoref{lem:Yos} if $E$ has good reduction at $p$) 
there is an exact sequence 
\[
0 \to \Ker(\dbar_{E,p})\to  \bigoplus_{l \in \Sigma_\bad(E)}V(E_l)/p \oplus V(E_{\R})/p
	\to E[p]_{G_\Q}\to \Coker(\dbar_{E,p})\to 0
\]
of finite dimensional $\Fp$-vector spaces.
As $p$ is odd, $V(E_{\R})/p = 0$ (\autoref{lem:real}). 
If $\dimFp(E[p]_{G_{\Q}}) = 2$, 
then $I(E[p]) = 0$ and $E[p]\subset E(\Q)$. 
By Mazur's theorem on 
the torsion subgroup $E(\Q)_{\tor}$ of $E(\Q)$ 
%says that $E(\Q)_{\tor}$ is isomorphic to one of the following 
%\begin{equation}
%	\label{eq:Maz}
%\Z/N\Z\quad \mbox{for $1\le N\le 10$ or $N=12$},\qquad \Z/2\Z\times \Z/2N\Z\quad \mbox{for $1\le N\le 4$}
%\end{equation}
(\cite[Thm.~2]{Maz78}, \Cf\cite[Thm.~7.5]{Sil106}),   
there is no odd prime $p$ satisfying $E[p] \subset E(\Q)$. 
From this reason, $\dimFp(E[p]_{G_{\Q}}) = 1$. 

Finally, in the case $(\mathrm{SC}_p)$, 
$\Coker(\dbar_{E,p}) = 0$ (\autoref{lem:SCp}). 
\end{proof}

\begin{ex}[Non-split multiplicative]\label{ex:35a1}
Consider the isogeny class of elliptic curves with conductor $35$ consisting of 3 semi-stable elliptic curves 
\[
\xymatrix@R=0mm{
E^{(2)} \ar[r]^{\phi} & E^{(1)} & \ar[l]_{\psi} E^{(3)} \\
\text{35a2} & \text{35a1} & \text{35a3}
}
\]
with isogenies $\phi$ and $\psi$ of degree $3$.  
The Mordell-Weil groups are $E^{(1)}(\Q)\simeq E^{(3)}(\Q) \simeq \Z/3$ and $E^{(2)}(\Q) = 0$ 
(\Cf\cite[\href{https://beta.lmfdb.org/EllipticCurve/Q/35/a/}{Elliptic Curve 35.a}]{lmfdb}). 
As $E^{(3)}$ satisfies $(\mathrm{B}'_3)$, $\dim_{\F_3}(E^{(3)}[3]_{G_{\Q}}) = 0$ (\autoref{lem:quot}). 

The curve $E^{(1)}$ satisfies $(\mathrm{SC}_3)$.  
We have $\Delta(E_{1}) = -1 \cdot 5^3 \cdot 7^3$, 
and $E^{(1)}$ has 
split multiplicative reduction at $7$ 
and non-split multiplicative reduction at $5$. 
For $l=7$, 
the Tate parameter is $q_7=7^3+4\cdot 7^4+ \cdots$ 
and hence 
$\dim_{\F_3}(V( E^{(1)}_7)/3) = 1$, 
where $E^{(1)}_7 := E^{(1)}\otimes_\Q \Q_7$ (\Cf\autoref{rem:tame}). 
For $l =5$, we only have an inequality $\dim_{\F_3}(V(E^{(1)}_5)/3)\le 1$. 
By \autoref{thm:EQ}, the map 
\[
\dbar_{E^{(1)},3}\colon V(E^{(1)})/3 \to \bigoplus_{l\in \Sigma_{\good}(E^{(1)})}(\ol{E^{(1)}_l})(\F_l)/3
\]
is surjective with $\dim_{\F_3}(\Ker(\dbar_{E,3})) \le 1$. 

The curve $E^{(2)}$ satisfies $(\mathrm{B}_3)$. 
We have $\Delta(E^{(2)}) = -5^9\cdot 7$ 
and $E^{(2)}$ has split multiplicative reduction at $7$ 
and non-split multiplicative reduction at $5$. 
For the bad prime $l=7$, $v_7(\Delta(E^{(2)})) = 1$ implies $\dim_{\F_3}(V(E^{(2)}_7)/3) = 0$ (\Cf\autoref{rem:tame}). 
For $l=5$, 
%the base change of $E^{(2)}$ to the quadratic unramified extension $K = \Q_5(\sqrt{\gamma})$ of $\Q_5$ 
%has split multiplicative reduction, 
%where $\gamma = -c_4/c_6 =  -788/65449 = 3 + 2\cdot5 + 3\cdot 5^3  +\cdots $ in $\Q_5$ 
%(\cite[Chap.~V, Thm.~5.3]{Sil151}). 
we have $\dim_{\F_3}(V( E^{(2)}_5)/3)\le 1$ (\autoref{lem:nonsp}).  
Inequalities $\dim_{\F_3}(\Ker(\dbar_{E^{(2)},3}))\le 1$ and $\dim_{\F_3}(\Coker(\dbar_{E^{(2)},3}))\le 1$ hold. 
\end{ex}

\begin{ex}[{Non-trivial $\Q$-rational 2-torsion}]\label{ex:17a2}
%% https://beta.lmfdb.org/EllipticCurve/Q/17a1/
Let $E$ be an elliptic curve over $\Q$ defined by 
\[
y^2+xy+y=x^3-x^2-6x-4
\]
(the Cremona label 17a2, \Cf \cite[\href{https://beta.lmfdb.org/EllipticCurve/Q/17/a/2}{Elliptic Curve 17.a2}]{lmfdb}) 
The Mordell-Weil group is $E(\Q) \simeq \Z/2\oplus \Z/2$, 
$\rho_p$ is surjective for all $p\neq 2$, and $\Delta(E) = 17^2$. 
For the prime $p = 2$, 
we have $\dim_{\F_2}(E[2]_{G_{\Q}}) = 2$ by \autoref{lem:even}. 
The elliptic curve $E$ has good reduction outside $17$ and 
has split multiplicative reduction at $17$. 
The Tate parameter $q_{17} \in \Q_{17}$ is of the form 
\[
q_{17} = 17^2 + 3\cdot 17^3 + \cdots. 
\]  
By similar arguments in \autoref{rem:tame}, $\dim_{\F_2}(V(E_{17})/2) = 1$. 
\autoref{lem:real} gives an inequality $\dim_{\F_2}(V(E_{\R})/2) \le 2$. 
Furthermore, 
the local boundary map at $2$ is surjective so that $\Coker(\dbar_{E,2}) = 0$ (\autoref{lem:SCp}). 
%The elliptic curve $E$ has good ordinary reduction at $2$ 
%and $E[2]\subset E(\Q)\subset E(\Q_2)$. 
%
As a consequence,  \autoref{thm:fin} gives an exact sequence  
\[
0\to \Ker(\dbar_{E,2}) \to \Ker(\dbar_{E_2,2})\oplus (\Z/2) \oplus V(E_{\R})/2 \to (\Z/2)^{\oplus 2} \to 0
\]
and $\dim_{\F_2}(\Ker(\dbar_{E,2}))\le 2$.
\end{ex}

%
%
%For a place $v\mid p$, the reduction map gives 
%$T_p(E) \surj \ol{E_v}[p]$. 
%This induces $T_p(E)_{G_F} \surj \ol{E_v}(\F_v)[p]$. 
%$\ol{E_v}(\F_v)[p] \neq \emptyset$ implies $m\mid p$. 
%
%
%$T_p(E)_{G_F}\twoheadleftarrow T_p(E)_{D_v} \simeq \piab(E_v)^{\geo} \simeq V(E_v)/V(E_v)_{\div}$ gives
%$(T_p(E)_{G_F})/p \twoheadleftarrow V(E_v)/p$ 
%If $F= \Q$, $V(E_p)/p \simeq \ol{E_p}(\Fp)/p$. 
%$\ol{E_p}(\Fp)[p] = 0$ implies $m=1$ and $V(E)/p =0$. 

%
%
%If we assume $E$ has multiplicative reduction at $v$, 
%
%or an \emph{odd} prime number $p$ 
%and an infinite place $v$ in $F$, 
%
%
%
%In the following, we consider $F=\Q$ and an odd prime number $p$. 
%
%For each place $v$ in $\Q$,  and hence  
%\[
%V(E_v)/p  \simeq \begin{cases}
%	0 & \mbox{if $v$ is finite and $E_v$ has multiplicative reduction},\\
%	\ol{E}_v(\F_v)/p& \mbox{if $v$ is finite and $E_v$ has good ordinary reduction},\\
%	? & \mbox{if $v$ is finite and $E_v$ has good supersingular reduction},\\
%	0 & \mbox{if $v$ is infinite.}
%\end{cases}
%\]

%\bibliographystyle{amsalpha}
%\bibliography{bound}

\begin{thebibliography}{{LMF}25}

\bibitem[Asa06]{Asa06}
M.~Asakura, \emph{Surjectivity of {$p$}-adic regulators on {$K_2$} of {T}ate
  curves}, Invent. Math. \textbf{165} (2006), no.~2, 267--324.

\bibitem[Blo81]{Blo81}
S.~Bloch, \emph{Algebraic {$K$}-theory and classfield theory for arithmetic
  surfaces}, Ann. of Math. (2) \textbf{114} (1981), no.~2, 229--265.

\bibitem[FV02]{FV02}
I.~B. Fesenko and S.~V. Vostokov, \emph{Local fields and their extensions},
  second ed., Translations of Mathematical Monographs, vol. 121, American
  Mathematical Society, Providence, RI, 2002, With a foreword by I. R.
  Shafarevich.

\bibitem[GH21]{GH21}
E.~Gazaki and T.~Hiranouchi, \emph{Divisibility results for zero-cycles},
  European Journal of Math (2021), 1--44.

\bibitem[GH23]{GH23}
\bysame, \emph{Abelian geometric fundamental groups for curves over a
  {$p$}-adic field}, J. Th\'eor. Nombres Bordeaux \textbf{35} (2023), no.~3,
  905--946.

\bibitem[Hir19]{Hir19}
T.~Hiranouchi, \emph{Local torsion primes and the class numbers associated to
  an elliptic curve over {$\mathbb{Q}$}}, Hiroshima Math. J. \textbf{49}
  (2019), no.~1, 117--127.

\bibitem[Hir22]{Hir22}
\bysame, \emph{Galois symbol map for a {T}ate curve}, Bull. Kyushu Inst.
  Technol. Pure Appl. Math. (2022), no.~69, 1--6.

\bibitem[Kat86]{Kat86b}
K.~Kato, \emph{Milnor {$K$}-theory and the {C}how group of zero cycles},
  Applications of algebraic {$K$}-theory to algebraic geometry and number
  theory, {P}art {I}, {II} ({B}oulder, {C}olo., 1983), Contemp. Math., vol.~55,
  Amer. Math. Soc., Providence, RI, 1986, pp.~241--253.

\bibitem[KL81]{KL81}
N.~M. Katz and S.~Lang, \emph{Finiteness theorems in geometric classfield
  theory}, Enseign. Math. (2) \textbf{27} (1981), no.~3-4, 285--319 (1982),
  With an appendix by Kenneth A. Ribet.

\bibitem[KS83]{KS83b}
K.~Kato and S.~Saito, \emph{Unramified class field theory of arithmetical
  surfaces}, Ann. of Math. (2) \textbf{118} (1983), no.~2, 241--275.

\bibitem[{LMF}25]{lmfdb}
The {LMFDB Collaboration}, \emph{The {L}-functions and modular forms database},
  \url{https://www.lmfdb.org}, 2025, [Online; accessed 14 March 2025].

\bibitem[Maz78]{Maz78}
B.~Mazur, \emph{Rational isogenies of prime degree (with an appendix by {D}.
  {G}oldfeld)}, Invent. Math. \textbf{44} (1978), no.~2, 129--162.

\bibitem[Mil06]{MilADT}
J.~S. Milne, \emph{Arithmetic duality theorems}, second ed., BookSurge, LLC,
  Charleston, SC, 2006.

\bibitem[Neu99]{Neu99}
J.~Neukirch, \emph{Algebraic number theory}, Grundlehren der Mathematischen
  Wissenschaften [Fundamental Principles of Mathematical Sciences], vol. 322,
  Springer-Verlag, Berlin, 1999, Translated from the 1992 German original and
  with a note by Norbert Schappacher, With a foreword by G. Harder.

\bibitem[NSW08]{NSW08}
J.~Neukirch, A.~Schmidt, and K.~Wingberg, \emph{Cohomology of number fields},
  second ed., Grundlehren der Mathematischen Wissenschaften [Fundamental
  Principles of Mathematical Sciences], vol. 323, Springer-Verlag, Berlin,
  2008.

\bibitem[Ram]{Ram25}
R.~Ramakrishnan, \emph{Global galois symbols on {$E\times E$}}, to appear in
  Indag. Math., arXiv:2407.20468.

\bibitem[Ras90]{Ras90}
W.~Raskind, \emph{On {$K_1$} of curves over global fields}, Math. Ann.
  \textbf{288} (1990), no.~2, 179--193.

\bibitem[RS00]{RS00}
W.~Raskind and M.~Spiess, \emph{Milnor {$K$}-groups and zero-cycles on products
  of curves over {$p$}-adic fields}, Compositio Math. \textbf{121} (2000),
  1--33.

\bibitem[RV01]{RV01}
A.~Reverter and N.~Vila, \emph{Images of mod {$p$} {G}alois representations
  associated to elliptic curves}, Canad. Math. Bull. \textbf{44} (2001), no.~3,
  313--322.

\bibitem[{Sag}24]{SAGE}
{Sage Developers}, \emph{{S}age{M}ath, the {S}age {M}athematics {S}oftware
  {S}ystem}, 2024.

\bibitem[Sai85]{Sai85a}
S.~Saito, \emph{Class field theory for curves over local fields}, J. Number
  Theory \textbf{21} (1985), no.~1, 44--80.

\bibitem[Ser68]{Ser68}
J.-P. Serre, \emph{Corps locaux}, Hermann, Paris, 1968, Deuxi\`eme \'edition,
  Publications de l'Universit\'e de Nancago, No. VIII.

\bibitem[Ser72]{Ser72}
\bysame, \emph{Propri\'et\'es galoisiennes des points d'ordre fini des courbes
  elliptiques}, Invent. Math. \textbf{15} (1972), no.~4, 259--331.

\bibitem[Ser96]{Ser96}
\bysame, \emph{Travaux de {W}iles (et {T}aylor, {$\ldots$}). {I}}, no. 237,
  1996, S\'eminaire Bourbaki, Vol.\ 1994/95, pp.~Exp. No. 803, 5, 319--332.

\bibitem[Sil09]{Sil106}
J.~H. Silverman, \emph{The arithmetic of elliptic curves}, second ed., Graduate
  Texts in Mathematics, vol. 106, Springer, Dordrecht, 2009.

\bibitem[Sil13]{Sil151}
\bysame, \emph{Advanced topic in the arithmetic of elliptic curves}, Graduate
  Texts in Mathematics, vol. 151, Springer, Dordrecht, 2013.

\bibitem[Som90]{Som90}
M.~Somekawa, \emph{On {M}ilnor {$K$}-groups attached to semi-abelian
  varieties}, $K$-Theory \textbf{4} (1990), no.~2, 105--119.

\bibitem[Wei05]{Wei05}
C.~Weibel, \emph{Algebraic {$K$}-theory of rings of integers in local and
  global fields}, Handbook of {$K$}-theory. {V}ol. 1, 2, Springer, Berlin,
  2005, pp.~139--190.

\bibitem[Yam05]{Yam05}
T.~Yamazaki, \emph{On {C}how and {B}rauer groups of a product of {M}umford
  curves}, Math. Ann. \textbf{333} (2005), 549--567.

\bibitem[Yos02]{Yos02}
T.~Yoshida, \emph{Abelian \'etale coverings of curves over local fields and its
  application to modular curves}, thesis (2002).

\end{thebibliography}

\def\cprime{$'$}
\providecommand{\bysame}{\leavevmode\hbox to3em{\hrulefill}\thinspace}
\providecommand{\href}[2]{#2}

\end{document}